\documentclass[11pt]{article}
\usepackage{amssymb}
\usepackage{latexsym}
\usepackage{amsfonts}
\usepackage{amsmath}
\newtheorem{thrm}{Theorem}[section]
\newtheorem{lemma}[thrm]{Lemma}
\newtheorem{prop}[thrm]{Proposition}

\newtheorem{remark}[thrm]{Remark}

\newtheorem{ques}{Question}

\numberwithin{equation}{section}
\usepackage[dvips]{color}
\usepackage{mathtools}
\mathtoolsset{showonlyrefs}
\usepackage{mathrsfs}
\usepackage{comment}
\usepackage{hyperref}
\usepackage{xcolor}

\makeatletter
\begin{document}
	\allowdisplaybreaks

	\title{\Large \bf On the maximal displacement of subcritical branching random walks with  stretched exponential tail}
	\author{ \bf   Haojie Hou
	}
	\date{}
	\maketitle
	
	\begin{abstract}
		We study the maximal displacement of a one-dimensional subcritical branching random walk with offspring distribution $\{p_k\}$ and step size $X$ such that $m := \sum_{k=1}^\infty k p_k \in (0,1)$. Let $M_n$ denote the maximal position of all particles alive at time $n$ and let $M := \sup_{n \in \mathbb{N}} M_n$. First, we show that
		\[
		\lim_{x \to +\infty} \frac{e^{\lambda x^b}}{\ell(x) x^a } \, \mathbb{P}(M > x) = \frac{1 - p_0}{1 - m}
		\]
		whenever $\mathbb{P}(X > x) = \ell(x) x^a e^{-\lambda x^b}$ for some slowly varying function $\ell$, $b \in [0,1)$, and under further assumptions on $a$. Next, we prove that
		\[
		\lim_{x \to +\infty} \frac{e^{\lambda x^b+\gamma x}}{\ell(x) x^a } \, \mathbb{P}(M > x) \quad \text{exists and belongs to } (0, \infty)
		\]
		provided that $\sum_{k=1}^\infty k (\log k) p_k < \infty$ and for some $x_*>0$, $\mathbb{P}(X > x) = \int_x^\infty \ell(y) y^a e^{-\lambda y^b - \gamma y} \, \mathrm{d}y$ for all $x > x_*$. Here, $\ell$ is a slowly varying function, $m \mathbb{E}(e^{\gamma X}) < 1$, $b \in [0,1)$, and $a$ satisfies certain conditions.
	\end{abstract}
	
	\medskip
	
	
	\medskip

	
\section{Introduction}

 \subsection{Model introduction and motivation}
 The branching random walk $(Z_n, n\in \mathbb{N}, \mathbb{P})$ is a discrete-time Markov process defined as follows. At time $0$, there is a single particle located at $0 \in \mathbb{R}$. At time $1$, this initial particle dies and produces a set of offspring, whose spatial configuration is an independent copy of a random point process $\mathcal{L}$. At time $2$, each particle existing at time $1$ independently repeats the behavior of its parent and each particle at site $x$ gives birth to an independent copy of $x+\mathcal{L}$. The procedure goes on. 
 
In this paper, our standard assumption is as follows.
\begin{itemize}
	\item[{\bf(H0)}] (First jump then branch) Assume that 	$\mathcal{L}:= N\delta_X$ where  $N $ is independent of $X$ and that the branching random walk is subcritical, i.e.,  $\mathbb{P}(N=k)=p_k$ and $m:= \mathbb{E}(N) =\sum_{k=1}^\infty kp_k \in (0,1)$.
\end{itemize}
Let $M_n$ denote the maximal position among all particles alive at time $n$, with the convention that $M_n := -\infty$ if there is no particle alive at time $n$. Since it is well known that in the subcritical case $m < 1$, the branching random walk dies out almost surely in finite time, the (all-time) maximal displacement 
\begin{align}\label{def-of-M}
	M := \sup_{n \in \mathbb{N}} M_n
\end{align}
is well-defined. One of the main points of interest in the subcritical branching random walk concerns the asymptotic behavior of the tail probability $\mathbb{P}(M > x)$ of $M$ as $x \to \infty$.

The first related work is due to Sawyer and Fleischman \cite{sawyer1979} in the setting of branching Brownian motion (with branching rate $1$). It was proved in \cite{sawyer1979} that, under the third moment condition $\sum_{k=1}^\infty k^3 p_k< \infty$, 
\begin{align}\label{Result-Sawyer}
	\lim_{x \to +\infty} e^{\sqrt{2(1-m)}x} \, \mathbb{P}(M > x)\ \mbox{exists and } \in (0,\infty). 
\end{align}
The main idea in \cite{sawyer1979} is based on a key observation such that the function $x \mapsto \mathbb{P}(M > x)$ solves a class of second-order ODE known as the F-KPP equation. Later, Profeta \cite{profeta24bernoulli} extended the Brownian case to spectrally negative L\'evy processes  and proved a similar result as \eqref{Result-Sawyer}  under third moment condition according to laplace transform method.

Now we introduce the related result in the discrete-time model. As for subcritical branching random walks, there have also been recent works concerning the tail of $M$.  Under condition {\bf(H0)} for  $\sum_{k=1}^\infty k^3p_k<\infty$ and some integer-valued centered random variable $X$ such that  $\mathbb{E}(e^{\theta X}) < \infty$ for all $\theta > 0$, Neuman and Zheng \cite{NZ2017} proved that
\begin{align} 
	0 < \liminf_{n \to \infty} e^{\gamma n} \, \mathbb{P}(M > n) \leq \limsup_{n \to \infty} e^{\gamma n} \, \mathbb{P}(M > n) \leq 1,
\end{align}
where $\gamma > 0$ is the unique solution to the equation 
\begin{align}\label{exists-of-gamma}
	m\mathbb{E}(e^{\gamma X}) = 1.
\end{align}
  Moreover, under some other technical assumptions on $X$ (i.e., $\mathbb{P}(X\leq A)=1$ for some $A<\infty$ and that $X$ is right-continuous), it was shown in Neuman and Zheng \cite{NZ2017} that 
  \begin{align}\label{Result-NZ-2}
  	\lim_{n\to\infty} e^{\gamma n} \, \mathbb{P}(M > n)\ \mbox{exists and } \in (0,1].
  \end{align} 
  After that, there are also some improvements to the result \eqref{Result-NZ-2}. 
  Based on a spine decomposition approach, Fu and Hong \cite{FH2025} proved \eqref{Result-NZ-2} under {\bf(H0)} and the following assumption:
  \begin{itemize}
  	\item[\bf (A1)] $ \sum_{k=1}^\infty k^2 p_k < \infty$; $\mathbb{E}(X) = 0$, $\mathbb{E}(e^{\theta X}) < \infty$ for all $\theta > 0$, and $\mathbb{P}(X \in \mathrm{d} x)$ is either a diffusion measure on $\mathbb{R}$ or a discrete and irreducible measure on $\mathbb{Z}$. 
  \end{itemize}
  Recently Hou and Zhang \cite{HZ2025} used renewal theory approach and proved \eqref{Result-NZ-2} under {\bf(H0)} and  the following assumption:
  \begin{itemize}
  	\item[\bf (A2)] $ \sum_{k=1}^\infty k(\log k)p_k < \infty$; \eqref{exists-of-gamma} and $\mathbb{E}(X e^{\gamma X}) < \infty$ hold. Either $X$ is non-lattice or for some $h>0$, $X$ is lattice with span $h$.
  \end{itemize}
The assumption {\bf (A2)}  only requires slightly stronger assumptions  than  {\bf(H0)} and the existence of $\gamma$ in \eqref{exists-of-gamma} and is weaker than {\bf (A1)}. The main motivation of this paper is based on the following natural question:
\begin{align}\label{Question}
	\emph{What can we say about $\mathbb{P}(M > x)$ if the equation \eqref{exists-of-gamma} does not admit a solution?}
\end{align}
The question \eqref{Question} reflects a transition between two different extremal mechanisms. In the classical Cramér-type regime, the asymptotic behavior of $M$ is driven by an exponential tilting associated with the solution of \eqref{exists-of-gamma}. Once this equation fails to have a solution, or the relevant exponential moment is not strong enough, the extremal behavior is no longer governed by such a tilted random-walk path. 
To the best of our knowledge, most of the literature addressing this question focuses only on the branching (strictly)$\alpha$-stable process for $\alpha \in (0,2)$; see, for example, Hou et al \cite{HJRS25} and Profeta \cite{profeta22alea}. It was proved in \cite{HJRS25, profeta22alea} that if the $\alpha$-stable process admits positive jumps, then under the subcritical assumption $m\in (0,1)$,
\begin{align}\label{Result-Profeta}
	\lim_{x \to +\infty} x^\alpha \, \mathbb{P}(M > x)\ \mbox{exists and } \in (0,\infty).
\end{align}
 The proofs in \cite{HJRS25, profeta22alea} are based on the Laplace transform argument and fine analysis on tail probability of an $\alpha$-stable process indexed by an independent exponential random variable, which are consistent with the strategies used in the case of the spectrally negative L\'evy process \cite{profeta24bernoulli}. The main goal of this paper is to understand the phase transition between \eqref{Result-NZ-2} and \eqref{Result-Profeta} and to partially answer the question \eqref{Question} for a class of jump sizes $X$.

\subsection{Assumptions and main results}
Consider the following two cases:
\begin{itemize}
	\item [{\bf(H1)}] The distribution of $X$ has a stretched exponential tail; that is, there exists a function $\ell$ slowly varying at infinity such that for all $x > 1$, 
	\begin{align}\label{Tail-of-X}
		\mathbb{P}(X > x) = \ell(x) x^a e^{-\lambda x^b},
	\end{align}
	where $a \in \mathbb{R}$, $\lambda > 0$ if $b \in (0,1)$, and $a < 0$, $\lambda = 0$ if $b = 0$. We also assume that $\mathbb{E}(X^2) < \infty$ when $b \in (0,1)$. 
	
	\item [{\bf(H2)}] There exists $\gamma, x_* > 0$ and a function $\ell$ slowly varying at infinity such that $m \mathbb{E}(e^{\gamma X}) < 1$  and that for all $x > x_*$, 
	\begin{align}
		\mathbb{P}(X > x) = \int_x^\infty \ell(y) y^a e^{-\lambda y^b - \gamma y} \, \mathrm{d}y, 
	\end{align}
	where $a \in \mathbb{R}$, $\lambda > 0$ if $b \in (0,1)$, and $a < -2$, $\lambda = 0$ if $b = 0$.
\end{itemize}

We mention here that the assumptions $\mathbb{E}(X^2)<\infty$  for $b\in (0,1)$ in {\bf(H1)}  and  $a<-2$ in {\bf(H2)} are needed for some technical reasons. We will explain these assumptions in suitable place  of the paper when we used them. 
Our main results are as follows.

	\begin{thrm}\label{thm1}
		Assume {\bf(H0)} and {\bf(H1)}. Then 
		\[
		\lim_{x\to+\infty} \frac{e^{\lambda x^b}}{\ell(x)x^a } \mathbb{P}(M>x) =\frac{1-p_0}{1-m}. 
		\]
	\end{thrm}

	\begin{thrm}\label{thm2}
	Assume {\bf(H0)},  {\bf(H2)} and $\sum_{k=1}^\infty k(\log k)p_k<\infty$. Then there exists a constant $C_0\in (0,\infty)$ given as in \eqref{expression-of-C} below such that 
	\[
	\lim_{x\to+\infty} \frac{e^{\lambda x^b+ \gamma x}}{\ell(x) x^a }\mathbb{P}(M>x) =C_0. 
	\]
\end{thrm}

\begin{remark}
	  For Theorem \ref{thm1}, under {\bf(H1)} with $b = 0$ and $a \in (-2, 0)$, our first result coincides with the case of the $\alpha$-stable process in \eqref{Result-Profeta}. Moreover, Theorem \ref{thm1} is equivalent to
	  \begin{align}\label{equi-thm1}
	  	\lim_{x \to +\infty} \frac{\mathbb{P}(M > x)}{\mathbb{P}(X > x)} = \frac{1 - p_0}{1 - m}.
	  \end{align}
	  For Theorem \ref{thm2}, the assumption $\sum_{k=1}^\infty k (\log k) p_k < \infty$ is used for technical purposes, and the result is equivalent to
	  \begin{align}\label{equi-thm2}
	  	\lim_{x \to +\infty} \frac{\mathbb{P}(M > x)}{\mathbb{P}(X \in \mathrm{d}x) / \mathrm{d}x} = C_0.
	  \end{align}
	  As we will discuss in Section \ref{S1.3} below, we adapt the same idea in the proof of both theorems, even though the results \eqref{equi-thm1} and \eqref{equi-thm2} may appear slightly different.
\end{remark}

We end this section with some questions. In \cite{HZ2025}, Hou and Zhang also considered the tail probability of $M$ for subcritical branching random walks with killing. Thus, our first question is:
\begin{ques} 
	Under assumptions {\bf(H0)} and either {\bf(H1)} or {\bf(H2)}, what is the tail probability of $M$ in subcritical branching random walks with killing? 
\end{ques}
In critical multitype branching random walks, Leh\'ericy \cite[Section 2.4]{Le2025} gave the decay rate of the tail probability for $\mathbb{P}(M > x)$. In our setting, we are also interested in subcritical multitype branching random walks, and our next question is:
\begin{ques}
	What can we say about $\mathbb{P}(M > x)$ in subcritical multitype branching random walks?
\end{ques}
In critical branching L\'evy processes where the spatial motion $(L_t, t \geq 0, \mathbb{P}_x)$ admits a negative drift $\mathbb{E}_0(L_1) < 0$, a similar sufficient exponential moment condition $\mathbb{E}(e^{\omega L_1}) = 1$ for some $\omega >0$ is assumed in Profeta \cite[Theorem 2]{Profeta2026}. We believe that the tail probability of $M$ would be different if $\{\omega>0: \mathbb{E}_0(e^{\omega L_1})<\infty\}\neq \emptyset$ but there is no solution to $\mathbb{E}(e^{\omega L_1}) = 1$. Thus our last question is:
\begin{ques}
	What is the decay rate of the tail probability $\mathbb{P}(M > x)$ in critical branching L\'evy processes when the L\'evy process does not have a sufficient exponential moment?
\end{ques}

 \subsection{Organization of the paper and proof strategies}\label{S1.3}

The organization of the paper is as follows. Define $T_x^+ := \inf\{n: S_n > x\}$. In Section \ref{S2}, we prove the decay rate of $\mathbb{P}(T_x^+ = n)$ under {\bf(H1)} (see Proposition \ref{prop1}). In Section \ref{S3}, under {\bf(H2)} and a suitable change of measure (see \eqref{Change-of-measure} below), we obtain a new random walk $(S_n, \widetilde{\mathbb{P}})$ and prove the decay rate of $\widetilde{\mathbb{E}}(e^{-\gamma S_n} \mathbf{1}_{\{T_x^+ = n\}})$ (see Proposition \ref{prop2}). Moreover, in Section \ref{S3}, we prove a renewal theorem for the random walk $(S_n, \widetilde{\mathbb{P}})$ in Proposition \ref{prop3}. The proofs of Theorems \ref{thm1} and \ref{thm2} are given in Section \ref{S4} and are presented in Sections \ref{S4.1} and \ref{S4.2}, respectively. Section \ref{Appendix} is devoted to the proof of Proposition \ref{prop3} for the case $b \in (0,1)$.

Our proof strategy differs from that used for subcritical branching $\alpha$-stable processes \cite{HJRS25, profeta22alea} and is inspired by the literature on a class of supercritical branching random walks, including branching random walks with heavy tails \cite{BHR2017, BHR2018, Durrett1983, RSZ2024}, branching random walks with stretched exponential tails \cite{DGH2023, Gantert2000}, and branching random walks outside the boundary case \cite{BHM2018, CH2026+}. In the aforementioned literature, a phenomenon known as the \emph{one-big-jump principle} (see \cite{BBY2024, DDS2008} for related results in the context of random walks) appears in these models. Accordingly, we adapt the observations from the supercritical setting to subcritical branching random walks and prove Theorems \ref{thm1} and \ref{thm2}.

We write $f(x) \lesssim g(x)$ to mean that there exists some constant $C$ such that $f(x) \leq C g(x)$ for all sufficiently large $x$. Similarly, $f(x) \lesssim_{\delta, \kappa, \dots} g(x)$ indicates that the constant $C$ may depend on the parameters $\delta, \kappa, \dots$. The notation $f(x) \gtrsim g(x)$ (respectively, $f(x) \gtrsim_{\delta, \kappa, \dots} g(x)$) is equivalent to $g(x) \lesssim f(x)$ (respectively, $g(x) \lesssim_{\delta, \kappa, \dots} f(x)$). We note that the constant $C$ may also depend on the slowly varying function $\ell$ as well as on the constants $m$, $p_k$, $a$, $\lambda$, $\gamma$, and $b$.

    \section{Random walks with stretched exponential  tail}\label{S2}
 	
 	 In this section, we assume that $X$ satisfies {\bf(H1)}. Let $\{X_i, i \in \mathbb{N}\}$ be i.i.d. random variables equal in law to $X$. Define
 	 \begin{align}\label{Def-of-stopping-time}
 	 	S_n := \sum_{i=1}^n X_i \quad \text{and} \quad T_x^+ := \inf\left\{n \in \mathbb{N} : S_n > x\right\}.
 	 \end{align}
 	 The main goal of this section is to prove the following proposition.
 
	\begin{prop}\label{prop1}
		\begin{itemize}
			\item[(i)] For any $x > 1$ and $n \in \mathbb{N}$,
			\[
			m^{n/2} \, \mathbb{P}\left(T_x^+ = n \right) \leq m^{n/2} \, \mathbb{P}\left(S_n > x \right) \lesssim \mathbb{P}(X > x).
			\]
			\item[(ii)] For each $n \in \mathbb{N}$,
			\[
			\lim_{x \to +\infty} \frac{\mathbb{P}(T_x^+ = n)}{\mathbb{P}(X > x)} = 1.
			\]
		\end{itemize}
	\end{prop}
	
	We first prove Proposition \ref{prop1} for the case $b = 0$. Before that, we gather some useful results in the following lemma.
	
	\begin{lemma}\label{lem-useful-fact}
		Let $\ell$ be a slowly varying function.
		
		\begin{itemize}
			\item[(i)] For any $0<z_1\leq z_2$, 
			$
			\lim_{x \to +\infty} \sup_{ y\in [z_1, z_2]} \big|  \ell(xy)/\ell(x) - 1\big| =0.
			$
			
			\item[(ii)] For any $A > 1$ and $\delta > 0$, there exists $K(A, \delta)$ such that
			\[
			\frac{\ell(y)}{\ell(x)} \leq A \max\left\{(y/x)^\delta, (y/x)^{-\delta}\right\} \quad \text{for all } x, y \geq K(A, \delta).
			\]
			Consequently, for any $\delta > 0$, we have
			\[
			x^{-\delta} \lesssim_\delta \ell(x) \lesssim_\delta x^{\delta}.
			\]
		\end{itemize}
	\end{lemma}
	\textbf{Proof: }  For (i), see \cite[Theorem 1.5.2, p.22]{BGT1987}. (ii) is known as Potter's theorem and can be found in \cite[Theorem 1.5.6, p.25]{BGT1987}

	\hfill$\Box$

\noindent
\textbf{Proof of Proposition \ref{prop1} for $b=0$:}  (i) The first inequality holds trivally, so we only prove the second inequality. 
If $x\leq n$, then according to the inequality $\mathbb{P}(X>n)= \ell(n)n^a\gtrsim m^{n/2}$, we get 
\begin{align}\label{e1}
	m^{n/2} \mathbb{P}\left(S_n>x\right)\leq m^{n/2}\lesssim   \mathbb{P}(X>n)\leq   \mathbb{P}(X>x). 
\end{align}
If $x>n$, then combining the union bound and Lemma \ref{lem-useful-fact} (ii) (with $A=2, \delta=1$),  we have $\ell(x/n)/\ell(x)\lesssim n$ and 
\begin{align} \label{e2}
	& m^{n/2}  \mathbb{P}\left(S_n> x \right)  \leq nm^{n/2}  \mathbb{P}(X>x/n) = nm^{n/2} (x/n)^a \ell(x/n) \nonumber\\
	&= m^{n/2} n^{1-a} \frac{\ell(x/n)}{\ell(x)} \mathbb{P}(X>x) \lesssim m^{n/2} n^{2-a} \mathbb{P}(X>x) \nonumber\\
	&\lesssim \mathbb{P}(X>x) .
\end{align}
Combining \eqref{e1} and \eqref{e2} implies (i).

 (ii) Since $\{T_x^+=1\}= \{X_1>x\}$, the limit holds trivally. We consider the case $n\geq 2$ here. 
 For each $\varepsilon\in (1/2,1)$, define 
 \begin{align}\label{def-of-E-n}
 E_n:= \left\{ \mbox{there exists at most one }i \leq n \mbox{ such that } X_i>x^\varepsilon \right\}.
 \end{align}
 Fixing any $\Gamma>0$ with $\Gamma(2\varepsilon+1)<-(2\varepsilon-1)a$. Since $x^{-\Gamma}\lesssim_\Gamma \ell(x)\lesssim_\Gamma x^\Gamma$ by Lemma \ref{lem-useful-fact} (ii),  
 \begin{align}\label{e4'}
 & \mathbb{P}\left(E_n^c\right) =\mathbb{P}\left(\exists i\neq j \mbox{ such that } X_i, X_j>x^\varepsilon \right)\leq n^2 \mathbb{P}(X>x^\varepsilon)^2 \nonumber\\
 &= n^2\frac{\left(\ell(x^\varepsilon)\right)^2}{\ell(x)} x^{(2\varepsilon-1)a}  \mathbb{P}(X>x) \lesssim_\Gamma n^2\frac{x^{2\varepsilon \Gamma}}{x^{-\Gamma}} x^{(2\varepsilon-1)a}  \mathbb{P}(X>x) \nonumber\\
 & = n^{2} x^{(2\varepsilon+1)\Gamma +(2\varepsilon-1)a} \mathbb{P}(X>x). 
 \end{align}
 Thus,    $\lim_{x\to +\infty} \frac{\mathbb{P}(E_n^c)}{\mathbb{P}(X>x)} =0$. 
 Since for $x$ large enough such that $x/n>x^{\varepsilon}$,  we have
 \[
 \{T_x^+=n\}\subset \{S_n>x\}\subset \{\exists\ 1\leq j\leq n\ s.t.\ X_j>x/n \}\subset \{\exists\ 1\leq j\leq n\ s.t.\ X_j>x^{\varepsilon} \}.
 \]
 Therefore, on $\{T_x^+=n\}\cap E_n$, there exists a unique $j$ such that $X_j>x^\varepsilon$ and that $\max_{\ell\neq j} X_\ell \leq x^{\varepsilon}$, which together with $S_n >x$ implies that $X_j> x-nx^\varepsilon>x^\varepsilon$ for large $x$. In conclusion, we have  
\begin{align}\label{e4}
	& \lim_{x\to +\infty} \frac{\mathbb{P}(T_x^+ =n)}{\mathbb{P}(X>x)}   =\lim_{x\to +\infty} \frac{1}{\mathbb{P}(X>x)}\mathbb{P}\left(T_x^+=n, \exists !\ 1\leq j\leq n\ \mbox{such that } X_j>x^\varepsilon \right) \nonumber\\
	&= \lim_{x\to +\infty} \frac{1}{\mathbb{P}(X>x)} \mathbb{P}\left(T_x^+=n, \exists \ 1\leq j\leq n\ \mbox{such that } X_j>x- nx^\varepsilon, \max_{\ell\neq j} X_\ell \leq x^\varepsilon \right). 
\end{align}
Now we would like to replace $\max_{\ell\neq j} X_\ell \leq x^\varepsilon $ by $\max_{\ell\neq j} |X_\ell| \leq x^\varepsilon $. Noticing that 
\begin{align}\label{e11-3}
	 & \lim_{x\to +\infty} \frac{1}{\mathbb{P}(X>x)}\mathbb{P}\left(\exists \ 1\leq j\leq n\ \mbox{such that } X_j>x- nx^\varepsilon, \max_{\ell\neq j} |X_\ell| > x^\varepsilon \right) \nonumber\\
	 &\leq \lim_{x\to +\infty} \frac{n}{\mathbb{P}(X>x)} \mathbb{P}\left(X_1>x- nx^\varepsilon, \max_{\ell\neq 1} |X_\ell| > x^\varepsilon \right)\nonumber\\
	 &=  \lim_{x\to +\infty} \frac{n}{\mathbb{P}(X>x)}\mathbb{P}\left(X>x- nx^\varepsilon\right)\mathbb{P}\left(  |X_\ell| > x^\varepsilon \right)^{n-1} = n \lim_{x\to +\infty} \mathbb{P}\left(  |X_\ell| > x^\varepsilon \right)^{n-1}  =0,
\end{align}
where the second equality follows from  Lemma \ref{lem-useful-fact} (i). Plugging this back to \eqref{e4} yields that 
\begin{align}\label{e5}
	& \lim_{x\to +\infty} \frac{\mathbb{P}(T_x^+ =n)}{\mathbb{P}(X>x)} \nonumber\\
	&  = \lim_{x\to +\infty} \frac{1}{\mathbb{P}(X>x)} \mathbb{P}\left(T_x^+=n, \exists \ 1\leq j\leq n\ \mbox{such that } X_j>x- nx^\varepsilon, \max_{\ell\neq j} |X_\ell | \leq x^\varepsilon \right). 
\end{align}
If the unique $j$ with $X_j>x-nx^\varepsilon$ and $\max_{\ell\neq j} |X_\ell|\leq x^\varepsilon$ satisfies $j\leq n-1$, then on $T_x^+=n$, it must hold that $X_j = S_{n-1}- \sum_{\ell\neq j ,\ell\leq n-1} X_\ell \leq x+ (n-2)x^\varepsilon\leq x+nx^\varepsilon $.  Therefore, 
\begin{align}\label{e5'}
	& \limsup_{x\to +\infty} \frac{1}{\mathbb{P}(X>x)}\mathbb{P}\left(T_x^+=n, \exists \ 1\leq j\leq n-1\ \mbox{such that } X_j>x- nx^\varepsilon, \max_{\ell\neq j} |X_\ell | \leq x^\varepsilon \right) \nonumber\\
	& \leq \limsup_{x\to +\infty} \frac{1}{\mathbb{P}(X>x)} \mathbb{P}\left( \exists 1\leq j\leq n-1\ \mbox{such that }  |X_j-x|\leq nx^\varepsilon\right) \nonumber\\
	&\leq n  \limsup_{x\to +\infty} \frac{1}{\mathbb{P}(X>x)}\mathbb{P}\left(   |X-x|\leq nx^\varepsilon\right) \nonumber\\
	&= n  \lim_{x\to +\infty} \frac{\ell(x-nx^\varepsilon)(x-nx^\varepsilon)^a-\ell(x+nx^\varepsilon)(x+nx^\varepsilon)^a}{\ell(x)x^a}=0,
\end{align}
where in the last equality we used  Lemma \ref{lem-useful-fact} (i). Now plugging this back to \eqref{e5} and noticing that 
$\{X_n>x- nx^\varepsilon, \max_{\ell \leq n-1} |X_\ell | \leq x^\varepsilon \}\subset  \{ \max_{1\leq i\leq n-1} S_i\leq x\}$ and $\{X_n>x+ nx^\varepsilon, \max_{1\leq \ell \leq n-1} |X_\ell | \leq x^\varepsilon \}\subset  \{ S_n > x\}$  for large $x$, by Lemma \ref{lem-useful-fact} (i), we conclude that 
\begin{align}
	&1= \lim_{x\to +\infty} \frac{1}{\mathbb{P}(X>x)} \mathbb{P}\left(X>x+ nx^\varepsilon\right) \mathbb{P}\left( |X| \leq x^\varepsilon \right)^{n-1}  \nonumber\\
	&= \lim_{x\to +\infty} \frac{1}{\mathbb{P}(X>x)} \mathbb{P}\left(X_n>x+ nx^\varepsilon, \max_{1\leq \ell \leq n-1} |X_\ell | \leq x^\varepsilon \right) \nonumber\\
	&\leq  \lim_{x\to +\infty} \frac{\mathbb{P}(T_x^+ =n)}{\mathbb{P}(X>x)}   = \lim_{x\to +\infty} \frac{1}{\mathbb{P}(X>x)}\mathbb{P}\left(S_n >x, X_n>x- nx^\varepsilon, \max_{1\leq \ell \leq n-1} |X_\ell | \leq x^\varepsilon \right) \nonumber\\
	& \leq  \lim_{x\to +\infty} \frac{\mathbb{P}\left( X_n>x- nx^\varepsilon \right)}{\mathbb{P}(X>x)} =1,
\end{align}
which implies the desired result for $b=0$.

\hfill$\Box$

Now we prove Proposition \ref{prop1} for $b \in (0,1)$. The proof idea is quite similar to the case $b = 0$, but requires more delicate estimates. Observe that if $g(x)$ is a positive function such that $\lim_{x \to +\infty} \frac{g(x)}{x^{1-b}} = 0$, then by Taylor's expansion, uniformly for all $|y| \leq g(x)$, we have $\lim_{x \to +\infty} ((x+y)^b - x^b) = 0$. Therefore, together with Lemma \ref{lem-useful-fact}(i), we obtain that uniformly for all $|y| \leq g(x)$,
\begin{align}\label{proper-of-tail}
	\lim_{x \to +\infty} \frac{\mathbb{P}(X > x + y)}{\mathbb{P}(X > x)} = \lim_{x \to +\infty} \frac{\ell(x + y) (x + y)^a}{\ell(x) x^a} e^{-\lambda ((x+y)^b - x^b)} = 1.
\end{align}
Set $\mu := \mathbb{E}(X)$, $\sigma := \sqrt{\operatorname{Var}(X)}$, and $\widehat{X} := \frac{X - \mu}{\sigma}$. Then, under {\bf(H1)}, 
\[
\mathbb{P}(\widehat{X} > x) = \mathbb{P}(X > \sigma x + \mu) = \ell(\sigma x + \mu) (\sigma x + \mu)^a e^{-\lambda (\sigma x + \mu)^b} =: \widehat{\ell}(x) x^a e^{-\lambda \sigma^b x^b}
\]
for some slowly varying function $\widehat{\ell}$ and for all $x > (1 - \mu)/\sigma$. Since we assumed $\mathbb{E}(X^2) < \infty$ in {\bf(H1)}, by \cite[Theorem 8.3]{DDS2008} (with $R(x) = \lambda \sigma^b x^b$), uniformly for all $n < \frac{1}{2} (\log z)^3$,
\[
\lim_{z \to +\infty} \sup_{n<\frac{1}{2}(\log z)^3} \left| \frac{\mathbb{P}(S_n > n\mu + \sigma z)}{n \mathbb{P}(X > \sigma z + \mu)} -1\right|=0.
\]
Replacing $n\mu + \sigma z$ by $x$ and noting that $n < (\log x)^3$ whenever $n < \frac{1}{2} (\log z)^3$ for sufficiently large $x$, it follows from \eqref{proper-of-tail} that
\begin{align}\label{e6}
	\lim_{x \to +\infty} \sup_{n < (\log x)^3} \left| \frac{\mathbb{P}(S_n > x)}{n \mathbb{P}(X > x)} - 1 \right| = 0.
\end{align}
 
\begin{lemma}\label{lem2}
	For $x>1$, define $\varepsilon_0:= \frac{1}{2}\min\{ b, 1-b\}$ and
	\[
	\theta(x):= \lambda x^{b-1}-\frac{(\log x)^2}{x}.
	\]
	 Then there exists a constant $C_*>0$ such that for large $x$,
	 \[
	 \mathbb{E}\left(e^{ \theta(x) \min\{ X-\mu, x-x^{\varepsilon_0} \} }\right)  \leq \exp\left\{C_* \theta^2(x)\right\}.
	 \]
\end{lemma}
\textbf{Proof: }  Set $x_0:=x-x^{\varepsilon_0}$ for simplicity. Then 
\begin{align}
	 & \mathbb{E}\left( e^{ \theta(x) \min\{ X-\mu, x_0 \} }\right) -1 = e^{\theta(x)x_0}\mathbb{P}(X\geq x_0+\mu)+\mathbb{E}\left(e^{\theta(x)(X-\mu)}1_{\{X-\mu <x_0\}}\right)-1\nonumber\\
	 & = e^{\theta(x)x_0}\mathbb{P}(X\geq x_0+\mu)+ \mathbb{E}\left(e^{\theta(x)(X-\mu)}1_{\{(\theta(x))^{-1}< X-\mu <x_0\}}\right)\nonumber\\
	 &\quad +\left(\mathbb{E}\left(e^{\theta(x)(X-\mu)}1_{\{X-\mu \leq \theta(x)^{-1}\}}\right)-1\right)\nonumber\\
	 &=:  I_1+I_2+I_3.
\end{align}

\noindent
{\bf Estimate for $I_1$.} From $\varepsilon_0< 1-b$ we have  $\lim_{x\to+\infty} (x_0^b - x^b )=  0$. Therefore, 
 combining $\ell(x)\lesssim x$ and $\lim_{x\to+\infty} \theta(x)(x-x_0)=  0$, it holds that 
\begin{align}\label{bound-for-I1}
	&I_1 = e^{\theta(x)x_0}\ell(x_0+\mu) |x_0+\mu|^a e^{-\lambda (x_0+\mu)^b}\lesssim e^{\theta(x)x} x^{|a|+1}e^{-\lambda x^b}\nonumber\\
	& = e^{-(\log x)^2} x^{|a|+1}\lesssim \theta^2(x),
\end{align}
where in the last inequality we used the fact that $\theta(x)\gtrsim x^{b-1}$.

\noindent
{\bf Estimate for $I_2$.} Combining $\ell(x)\lesssim x$ and $\lceil x_0 \rceil \leq \lfloor x\rfloor$, we get
\begin{align}
	& I_2 \leq \sum_{k= \lfloor (\theta(x))^{-1}\rfloor }^{\lceil x_0\rceil} e^{\theta(x)(k+1)} \mathbb{P}\left( X-\mu\in (k,k+1]\right)\leq \sum_{k= \lfloor (\theta(x))^{-1}\rfloor}^{\lceil x_0\rceil} e^{\theta(x)(k+1)} \mathbb{P}\left( X-\mu >k\right)\nonumber\\
	&\lesssim \sum_{k= \lfloor (\theta(x))^{-1}\rfloor}^{ \lfloor x\rfloor} e^{\theta(x)k} k^{a+1} e^{-\lambda k^b} \lesssim \int_{(\theta(x))^{-1}/2}^{x} e^{\theta(x) y} y^{a+1}e^{-\lambda y^b} \mathrm{d}y. 
\end{align}
Noticing that $\theta (x)y -\lambda y^b\leq  \lambda x^{b-1}y -\lambda y^b \leq -\lambda y^b(1-2^{b-1})$ for $y<x/2$ and 
\begin{align}
\theta (x)y -\lambda y^b & =  \lambda x^{b-1}y^b(y^{1-b}-x^{1-b})  - (\log x)^2 y/x  \leq - (\log x)^2 /2
\end{align}
for $x/2\leq y<x$,   we conclude that 
\begin{align}\label{bound-for-I2}
	& I_2 \lesssim \int_{(\theta(x))^{-1}/2}^{x/2}  y^{a+1}e^{-\lambda y^b(1-2^{b-1})} \mathrm{d}y  +\int_{x/2}^{x}  y^{a+1}e^{-(\log x)^2/2 } \mathrm{d}y\nonumber\\
	&\lesssim \theta^2(x)\int_{(\theta(x))^{-1}/2}^{\infty}  y^{a+3}e^{-\lambda y^b(1-2^{b-1})} \mathrm{d}y  + x^{a+2}e^{-(\log x)^2/2 } \lesssim \theta^2(x).
\end{align}

\noindent
{\bf Estimate for $I_3$.} Combining inequalities 
\[
\mathbb{E}\left(\theta(x)(X-\mu)1_{\{X-\mu\leq \theta(x)^{-1}\}}\right)=-\mathbb{E}\left(\theta(x)(X-\mu)1_{\{X-\mu> \theta(x)^{-1}\}}\right) <0
\]
and  $|e^{x}-1-x|\lesssim x^2$ for $x\leq 1$, we get 
\begin{align}\label{bound-for-I3}
	& I_3 = \mathbb{E}\left(\left(e^{\theta(x)(X-\mu)} -1-\theta(x)(X-\mu)\right)1_{\{X-\mu \leq \theta(x)^{-1}\}}\right)-\mathbb{P}(\theta(x)(X-\mu)>1)\nonumber\\
	&\qquad + \mathbb{E}\left(\theta(x)(X-\mu)1_{\{X-\mu\leq \theta(x)^{-1}\}}\right)\nonumber\\
	&\leq \mathbb{E}\left(\left|e^{\theta(x)(X-\mu)} -1-\theta(x)(X-\mu)\right|1_{\{X-\mu \leq \theta(x)^{-1}\}}\right)\nonumber\\
	&\lesssim \theta^2(x) .
\end{align}
Thus, combining \eqref{bound-for-I1}, \eqref{bound-for-I2} and \eqref{bound-for-I3} implies the desired result.

\hfill$\Box$

\noindent
\textbf{Proof of Proposition \ref{prop1} for $b\in (0,1)$:}
Let $\varepsilon_0$ be given as in Lemma \ref{lem2}. 

(i) We divide the proof into three cases.

\noindent
{\bf Case 1: $n<(\log x)^3$.}
 By \eqref{e6}, we have
  \begin{align}\label{e7}
		m^{n/2}\mathbb{P}(T_x^+=n)\leq m^{n/2} \mathbb{P}(S_n>x) \lesssim n m^{n/2} \mathbb{P}(X>x)\lesssim   \mathbb{P}(X>x).
\end{align}

\noindent
{\bf Case 2: $n\geq x^{(1+b)/2}$.}
Noticing that in this case, $n\log (1/m) /2 >2\lambda x^b$ when $x$ is large enough. Therefore, it follows from $\mathbb{P}(X>x)\gtrsim e^{-2\lambda x^b}$  that 
\begin{align}\label{e8}
	\mathbb{P}(X>x)\gtrsim e^{-2\lambda x^b} \geq m^{n/2} \quad  \Longrightarrow \quad m^{n/2}\mathbb{P}(T_x^+=n)\leq m^{n/2} \lesssim  	\mathbb{P}(X>x).
\end{align}

\noindent
{\bf Case 3: $(\log x)^3\leq n <x^{(1+b)/2}$.}
Let $x$ be large enough such that $x^{\varepsilon_0}>\mu$,  then by Markov inequality, 
\begin{align}\label{e33-2}
	&\mathbb{P}\left(S_n >x, \max_{1\leq i\leq n} X_i \leq x-2x^{\varepsilon_0} \right)  \leq \mathbb{P}\left(S_n-n\mu >x-n\mu, \max_{1\leq i\leq n} (X_i-\mu) \leq x-x^{\varepsilon_0} \right) \nonumber\\
	& \leq  \mathbb{P}\left(\sum_{i=1}^n \min\{ X_i-\mu, x-x^{\varepsilon_0}\}>x-n\mu \right)\nonumber\\
	& \leq e^{-\theta(x)(x-n\mu)}\left(\mathbb{E}\left( e^{\theta(x) \min\{ X-\mu, x-x^{\varepsilon_0}\} } \right)\right)^n.
\end{align}
Now it follows from Lemma \ref{lem2} that 
\begin{align}\label{e8-1}
	& \mathbb{P}\left(S_n >x, \max_{1\leq i\leq n} X_i \leq x-2x^{\varepsilon_0} \right) \leq e^{-\theta(x)(x-n\mu)} e^{C_*\theta^2(x)n}\nonumber\\
	&=  e^{-\lambda x^{b} +(\log x)^2+\left(C_*\theta^2(x)+ \mu \theta(x) \right)n}.
\end{align}
Since $\ell(x)x^a \gtrsim x^{-|a|-1}$ and  $\log x<n^{1/3}$, the above probability is bounded by
\begin{align}
	& \mathbb{P}\left(S_n >x, \max_{1\leq i\leq n} X_i \leq x-2x^{\varepsilon_0} \right)  \lesssim \mathbb{P}(X>x) e^{ (|a|+1)\log x +(\log x)^2+\left(C_*\theta^2(x)+ \mu \theta(x) \right)n}\nonumber\\
	& \leq \mathbb{P}(X>x) e^{ \left((|a|+1)n^{-2/3} +n^{-1/3}+C_*\theta^2(x)+ \mu \theta(x) \right)n}. 
\end{align}
Therefore, taking $x$ sufficiently large such that 
$(|a|+1)n^{-2/3} +n^{-1/3}+C_*\theta^2(x) + \mu \theta(x) \leq -\log m/2$ for all $n> (\log x)^3$, we conclude that 
\begin{align}\label{e9-1}
		m^{n/2}\mathbb{P}\left(S_n >x, \max_{1\leq i\leq n} X_i \leq x-2x^{\varepsilon_0} \right)&  \lesssim m^{n/2} \mathbb{P}(X>x) e^{-n\log m/2} =  \mathbb{P}(X>x) .
\end{align}
Next, it follows from \eqref{proper-of-tail} and  $\varepsilon_0<1-b$ that
\begin{align}\label{e9-2}
	& m^{n/2} \mathbb{P}\left(S_n >x,  \exists\  1\leq  i\neq j\leq n, X_i, X_j > x-2x^{\varepsilon_0} \right) \nonumber\\
	& \leq m^{n/2} n^2 \left(\mathbb{P}(X> x-2x^{\varepsilon_0})\right)^2 \lesssim   \left(\mathbb{P}(X> x)\right)^2 \leq  \mathbb{P}(X> x).
\end{align}
Combining \eqref{e9-1} and \eqref{e9-2}, we get that when $x$ is large enough, for $(\log x)^3\leq n< x^{(1+b)/2}$,
\begin{align}\label{e9}
	& m^{n/2}\mathbb{P}(T_x^+=n) \leq m^{n/2} \mathbb{P}(S_n>x) \nonumber\\
	&\lesssim   \mathbb{P}(X>x)   + n m^{n/2} \mathbb{P}\left(S_n>x, X_n>x-2x^{\varepsilon_0}, \max_{1\leq j\leq n-1} X_j\leq x-2x^{\varepsilon_0}\right) \nonumber\\
	&\leq    \mathbb{P}(X>x)  + n m^{n/2} \mathbb{P}\left( X_n>x-2x^{\varepsilon_0} \right) \lesssim   \mathbb{P}(X>x)  ,
\end{align}
where the last inequality follows from \eqref{proper-of-tail}.  Combining \eqref{e7}, \eqref{e8} and \eqref{e9}, we complete the proof of (i).

(ii) The case $n=1$ holds trivally, so we consider the case $n\geq 2$. On one hand, from \eqref{e9-2},   for each $n\in \mathbb{N}$,
\begin{align}\label{e11-1}
	\lim_{x\to+\infty} \frac{1}{\mathbb{P}(X>x)}\mathbb{P}\left(T_x^+=n, \exists 1\leq  i\neq j\leq n, X_i, X_j > x-2x^{\varepsilon_0} \right) =0. 
\end{align}
On the other hand, let $x$ be large such that $n< (\log x)^3$, then combining \eqref{proper-of-tail} and \eqref{e6},  it holds that for all $1\leq j\leq n$,
\begin{align}
 	& \lim_{x\to+\infty} \frac{1}{\mathbb{P}(X>x)} \mathbb{P}\left(S_j\in (x-2x^{\varepsilon_0}, x] \right) \nonumber\\
 	&= \lim_{x\to+\infty} \frac{\mathbb{P}\left(S_j> x-2x^{\varepsilon_0} \right)}{\mathbb{P}(X>x-2x^{\varepsilon_0})} \frac{\mathbb{P}\left(X> x-2x^{\varepsilon_0} \right)}{\mathbb{P}(X>x)} -\lim_{x\to+\infty} \frac{\mathbb{P}\left(S_j>x \right)}{\mathbb{P}(X>x)} =0.
\end{align}
Therefore,
\begin{align}
	& \limsup_{x\to+\infty} \frac{1}{\mathbb{P}(X>x)}\mathbb{P}\left(T_x^+=n, \max_{1\leq i\leq n} X_i \leq x-2x^{\varepsilon_0}\right) \nonumber\\
	& \leq  \limsup_{x\to+\infty} \frac{1}{\mathbb{P}(X>x)} \mathbb{P}\left(S_{n-1} \leq x-2x^{\varepsilon_0}, S_n> x,  X_n \leq x-2x^{\varepsilon_0}\right) + \lim_{x\to+\infty} \frac{\mathbb{P}\left(S_{n-1}\in (x-2x^{\varepsilon_0}, x] \right)}{\mathbb{P}(X>x)} \nonumber\\
	& \leq \limsup_{x\to+\infty} \frac{1}{\mathbb{P}(X>x)}\mathbb{P}\left(x-X_n\leq  S_{n-1},  2x^{\varepsilon_0}\leq X_n \leq x-2x^{\varepsilon_0}\right).
\end{align}
Since $S_{n-1}$ and $X_n$ are independent, combining \eqref{e6}, $\ell (x)\lesssim x$ and the fact that $x-X_n \geq 2x^{\varepsilon_0}$ on the event $\{2x^{\varepsilon_0}\leq X_n \leq x-2x^{\varepsilon_0}\}$,  it holds that 
\begin{align}\label{e10}
	& \limsup_{x\to+\infty} \frac{1}{\mathbb{P}(X>x)}\mathbb{P}\left(T_x^+=n, \max_{1\leq i\leq n} X_i \leq x-2x^{\varepsilon_0}\right)\nonumber\\
	& \leq (n-1) \limsup_{x\to+\infty} \frac{1}{\mathbb{P}(X>x)}\mathbb{E}\left(1_{\left\{2x^{\varepsilon_0}\leq X_n \leq x-2x^{\varepsilon_0} \right\}} \mathbb{P}(X \geq x-z)\big|_{z=X_n} \right) \nonumber\\
	&\lesssim  n \limsup_{x\to+\infty} \frac{1}{\mathbb{P}(X>x)} \mathbb{E}\left(1_{\left\{2x^{\varepsilon_0}\leq X_n \leq x-2x^{\varepsilon_0} \right\}}  |x-X_n|^{|a|+1}e^{-\lambda (x-X_n)^b} \right) .
\end{align}
Noticing that $\mathbb{P}(X_n \in (k, k+1])\leq \mathbb{P}(X_n >k)\lesssim k^{|a|+1}e^{-\lambda k^b}$, we have
\begin{align}
	&\mathbb{E}\left(1_{\left\{2x^{\varepsilon_0}\leq X_n \leq x-2x^{\varepsilon_0} \right\}}  |x-X_n|^{|a|+1}e^{-\lambda (x-X_n)^b} \right) \nonumber\\
	&\leq x^{|a|+1} \sum_{k=\lfloor 2x^{\varepsilon_0}\rfloor}^{\lceil x-2x^{\varepsilon_0}\rceil} e^{-\lambda (x-k-1)^b}\mathbb{P}(X_n\in (k, k+1]) \nonumber\\
	&\lesssim x^{|a|+1} \sum_{k=\lfloor 2x^{\varepsilon_0}\rfloor}^{\lceil x-2x^{\varepsilon_0}\rceil}   e^{-\lambda (x-k-1)^b}k^{|a|+1}e^{-\lambda k^{b}}\nonumber\\
	&\leq x^{2|a|+2} \sum_{k=\lfloor 2x^{\varepsilon_0}\rfloor}^{\lceil x-2x^{\varepsilon_0}\rceil}   e^{-\lambda (x-k-1)^b-\lambda k^{b}}.
\end{align}
Since $k^b + (x-k-1)^b\geq  \min_{y\in \{ \lfloor 2x^{\varepsilon_0}\rfloor,  \lceil x-2x^{\varepsilon_0} \rceil \}}(y^b + (x-y-1)^b) \geq (2x^{\varepsilon_0})^b +x^b-1$ for large $x$, we conclude that 
\begin{align}\label{e10-1}
&\mathbb{E}\left(1_{\left\{2x^{\varepsilon_0}\leq X_n \leq x-2x^{\varepsilon_0} \right\}}  |x-X_n|^{|a|+1}e^{-\lambda (x-X_n)^b} \right) \nonumber\\
&\lesssim x^{2|a|+3} e^{-\lambda x^b-\lambda (2x^{\varepsilon_0})^{b}} \lesssim x^{3|a|+4}e^{-\lambda (2x^{\varepsilon_0})^{b}} \mathbb{P}(X>x).
\end{align}
Plugging this back to \eqref{e10} implies that 
\begin{align}\label{e11-2}
	& \lim_{x\to+\infty} \frac{1}{\mathbb{P}(X>x)} \mathbb{P}\left(T_x^+=n, \max_{1\leq i\leq n} X_i \leq x-2x^{\varepsilon_0}\right)= 0. 
\end{align}
Combining \eqref{e11-1} and \eqref{e11-2}, we obtain that
\begin{align}\label{e11-4}
	& \lim_{x\to+\infty} \frac{\mathbb{P}\left(T_x^+=n\right)}{\mathbb{P}(X>x)} \nonumber\\
	&= \lim_{x\to+\infty} \frac{1}{\mathbb{P}(X>x)}\mathbb{P}\left(T_x^+=n, \exists !\ 1\leq j\leq n\ \mbox{such that } \ X_j > x-2x^{\varepsilon_0}\right).
\end{align}
Similar to \eqref{e11-3}, by \eqref{proper-of-tail}, we have the following estimate:
\begin{align}
	 & \lim_{x\to +\infty} \frac{1}{\mathbb{P}(X>x)}\mathbb{P}\left(\exists\  1\leq j\leq n\ \mbox{such that } X_j>x- 2x^{\varepsilon_0}, \max_{\ell\neq j} |X_\ell| > x^{\varepsilon_0} \right) \nonumber\\
	&\leq \lim_{x\to +\infty} \frac{n}{\mathbb{P}(X>x)}\mathbb{P}\left(X_1>x- 2x^{\varepsilon_0}, \max_{\ell\neq 1} |X_\ell| > x^{\varepsilon_0} \right) \nonumber\\
	&=  \lim_{x\to +\infty} \frac{n}{\mathbb{P}(X>x)} \mathbb{P}\left(X>x- 2x^{\varepsilon_0}\right)\mathbb{P}\left(  |X_\ell| > x^{\varepsilon_0} \right)^{n-1}= n \lim_{x\to +\infty} \mathbb{P}\left(  |X_\ell| > x^{\varepsilon_0 }\right)^{n-1}  =0.
\end{align}
Therefore, the limit on the left hand side of \eqref{e11-4} is equal to 
\begin{align}\label{e11-5}
	& \lim_{x\to+\infty} \frac{\mathbb{P}\left(T_x^+=n\right)}{\mathbb{P}(X>x)} \nonumber\\
	& = \lim_{x\to+\infty} \frac{1}{\mathbb{P}(X>x)}\mathbb{P}\left(T_x^+=n, \exists \ 1\leq j\leq n \mbox{ such that } X_j > x-2x^{\varepsilon_0}, \max_{ \ell\neq j} |X_\ell| \leq x^{\varepsilon_0}\right).
\end{align}
If the unique large jump $X_j$ satisfies $j\leq n-1$, then on $\{T_x^+=n \}\cap \{\max_{\ell\neq j} |X_\ell| \leq x^{\varepsilon_0}\}$, it holds that $X_j =S_j- \sum_{\ell=1}^{j-1}X_\ell \leq x+ nx^{\varepsilon_0}$, which together with \eqref{proper-of-tail} implies that 
\begin{align}\label{e11-6}
	&\lim_{x\to+\infty} \frac{1}{\mathbb{P}(X>x)}\mathbb{P}\left(T_x^+=n, \exists\  1\leq j\leq n-1 \mbox{ such that } X_j > x-2x^{\varepsilon_0}, \max_{ \ell\neq j} |X_\ell| \leq x^{\varepsilon_0}\right) \nonumber\\
	&\leq \sum_{j=1}^{n-1}\lim_{x\to+\infty} \frac{1}{\mathbb{P}(X>x)}\mathbb{P}\left( x+ nx^{\varepsilon_0} \geq X_j > x-2x^{\varepsilon_0}\right) =0.
\end{align}
Therefore, combining \eqref{e11-5} and \eqref{e11-6}, we have the following upper bound
\begin{align}\label{eq12-1}
	& \lim_{x\to+\infty} \frac{\mathbb{P}\left(T_x^+=n\right)}{\mathbb{P}(X>x)}  = \lim_{x\to+\infty} \frac{1}{\mathbb{P}(X>x)}\mathbb{P}\left(T_x^+=n,  X_n > x-2x^{\varepsilon_0}, \max_{1\leq  \ell \leq n-1} |X_\ell | \leq  x^{\varepsilon_0}\right) \nonumber\\
	&\leq \lim_{x\to+\infty} \frac{\mathbb{P}\left(X_n > x-2x^{\varepsilon_0}\right)}{\mathbb{P}(X>x)}=1,
\end{align}
where the last equality follows from \eqref{proper-of-tail}. For the lower bound, noticing that on the event $\{ X_n > x+nx^{\varepsilon_0},\max_{ 1\leq \ell \leq n-1} |X_\ell| \leq x^{\varepsilon_0}\}$, we have $\max_{1\leq j\leq n-1} S_j \leq nx^{\varepsilon_0} \leq x$ and $S_n =X_n+S_{n-1}>x$ for large $x$. Therefore,
by \eqref{proper-of-tail},   
\begin{align}\label{eq12-2}
	& \lim_{x\to+\infty} \frac{\mathbb{P}\left(T_x^+=n\right)}{\mathbb{P}(X>x)} \geq \lim_{x\to+\infty} \frac{1}{\mathbb{P}(X>x)}\mathbb{P}\left(X_n > x+nx^{\varepsilon_0},\max_{ 1\leq \ell \leq n-1} |X_\ell| \leq x^{\varepsilon_0} \right) \nonumber\\
	&= \lim_{x\to+\infty} \frac{1}{\mathbb{P}(X>x)}\mathbb{P}\left(X> x+nx^{\varepsilon_0}\right) \mathbb{P}\left( |X| \leq x^{\varepsilon_0} \right)^{n-1}=1.
\end{align}
Combining \eqref{eq12-1} and \eqref{eq12-2} completes the proof of (ii).

\hfill$\Box$

\section{Random walks without sufficient exponential moment}\label{S3}

 In this section, we assume that $X$ satisfies {\bf(H2)}. According to the representation theorem (for example, see \cite[Theorem 1.3.1, p.12]{BGT1987}), there exists $K_0>x_*$  such that 
 \begin{align}\label{Def-of-K-0}
 	0< \inf_{x\in [K_1, K_2]} \ell(x) \leq \sup_{x\in [K_1, K_2]} \ell(x)<\infty \quad\mbox{for all}\quad K_2>K_1\geq K_0.
 \end{align}
 Let $\{X_i, i\in\mathbb{N}\}$ be iid random variables equal in law to $X$. Recall the definitions of $S_n$ and $T_x^+$ in \eqref{Def-of-stopping-time}.  For each $n$, define 
 \begin{align}\label{Change-of-measure}
 	\frac{\mathrm{d} \widetilde{\mathbb{P}}}{\mathrm{d} \mathbb{P}}\bigg|_{\sigma(X_1,...,X_n)}:= \frac{1}{\big(\mathbb{E}(e^{\gamma X})\big)^n} e^{\gamma S_n}.
 \end{align}
 According to {\bf(H2)}, for any $x>x_*$,
 \begin{align}\label{density-X-tildeP}
 	\widetilde{\mathbb{P}}(X>x)= \int_x^\infty \frac{\ell(y)}{\mathbb{E}(e^{\gamma X})} y^a e^{-\lambda y^b}\mathrm{d}y.
 \end{align}
We check here that $(X, \widetilde{\mathbb P})$ satisfies {\bf(H1)}. 

\begin{lemma}\label{lem1}
	There exists some slowly varying function $\widetilde{\ell}$ at $\infty$ such that for all $x>K_0$, 
	\[
		\widetilde{\mathbb{P}}(X>x)=	\widetilde{\ell}(x) x^{a+1-b} e^{-\lambda x^b}.
	\]
	Moreover, $\widetilde{\mathbb{E}}(|X|^k)<\infty$ for all $k>0$ when $b\in (0,1)$ and $\lim_{x\to+\infty} \frac{\widetilde{\ell}(x)}{\ell(x)}= \frac{-(a+1)}{\mathbb{E}(e^{\gamma X})}1_{\{b=0\}}+ \frac{1}{\lambda b \mathbb{E}(e^{\gamma X})}1_{\{b\in (0,1)\}}$.
\end{lemma}
\textbf{Proof: } We divide the proof into two steps.

\noindent
{\bf (Step 1).} In this step, we prove the lemma for $b=0$. From \cite[Proposition 1.5.10, p. 27]{BGT1987}, we have
\begin{align} 
	\lim_{x\to+\infty} \frac{\widetilde{\mathbb{P}}(X>x)}{x^{a+1}\ell(x)}= \frac{-(a+1)}{\mathbb{E}(e^{\gamma X})},
\end{align}
which completes the proof of the lemma by taking $\widetilde{\ell}(x) =x^{-(a+1)}\widetilde{\mathbb{P}}(X>x)$.

\noindent
{\bf (Step 2).} In this step, we prove the lemma for $b\in (0,1)$. Noticing that 
\begin{align}
	\int_x^\infty y^a e^{-\lambda y^b}\mathrm{d} y =  \frac{1}{b}\int_{x^b}^\infty z^{\frac{a+1-b}{b}} e^{-\lambda z}\mathrm{d}z = \frac{x^{a+1-b}e^{-\lambda x^b}}{b}\int_{0}^\infty (1+zx^{-b})^{\frac{a+1-b}{b}} e^{-\lambda z}\mathrm{d}z .
\end{align}
According to dominated convergence theorem,  we get
\begin{align}
	\lim_{x\to+\infty} \frac{\int_x^\infty y^a e^{-\lambda y^b}\mathrm{d} y }{x^{a+1-b}e^{-\lambda x^b}}= \frac{1}{b}\int_0^\infty e^{-\lambda z}\mathrm{d}z = \frac{1}{\lambda b}.
\end{align}
Consequently, with $x$ replaced by $2x$ in the above limit, we obtain 
\begin{align}\label{Limit2}
			\lim_{x\to+\infty} \frac{\int_x^{2x} y^a e^{-\lambda y^b}\mathrm{d} y }{x^{a+1-b}e^{-\lambda x^b}} = \frac{1}{\lambda b}.
\end{align}
According to Lemma \ref{lem-useful-fact} (ii) (with $A=2, \delta=1$), 
\begin{align}\label{e14}
		\widetilde{\mathbb{P}}(X>2x) & =\frac{\ell(x)}{\mathbb{E}(e^{\gamma X})}  \int_{2x}^\infty \frac{\ell(y)}{\ell(x)}y^a e^{-\lambda y^b}\mathrm{d}y\lesssim \ell(x)  \int_{2x}^\infty \frac{y^{a+1}}{x} e^{-\lambda y^b}\mathrm{d}y.
\end{align}
Therefore,  combining  \eqref{e14} and L'H\^opital's rule, we conclude that 
 \begin{align}
 	\limsup_{x\to+\infty} \frac{\widetilde{\mathbb{P}}(X>2x) }{\ell(x) x^{a+1-b} e^{-\lambda x^b}} & \lesssim \lim_{x\to+\infty}\frac{\int_{2x}^\infty y^{a+1} e^{-\lambda y^b}\mathrm{d}y}{x^{a+2-b} e^{-\lambda x^b}} \nonumber\\
 	& = \lim_{x\to+\infty}\frac{-2 (2x)^{a+1} e^{-\lambda (2x)^b} }{ (a+2-b)x^{a+1-b} e^{-\lambda x^b} - \lambda bx^{a+2}e^{-\lambda x^b}} =0.
 \end{align}
 Since $\ell(y)/\ell(x)\to 1$ uniformly for $y\in [x,2x]$ as $x\to+\infty$, combining \eqref{Limit2} and the above limit, it holds that 
 \begin{align}
 		 & \lim_{x\to+\infty} \frac{\widetilde{\mathbb{P}}(X>x) }{\ell(x) x^{a+1-b} e^{-\lambda x^b}}  = 	\lim_{x\to+\infty} \frac{\widetilde{\mathbb{P}}(X>x) - \widetilde{\mathbb{P}}(X>2x) }{\ell(x) x^{a+1-b} e^{-\lambda x^b}}\nonumber\\
 		&= \frac{1}{ \lambda b \mathbb{E}(e^{\gamma X})}\lim_{x\to \infty} \frac{ \int_x^{2x} \ell(y) y^a e^{-\lambda y^b}\mathrm{d}y}{\ell(x) \int_x^{2x} y^a e^{-\lambda y^b}\mathrm{d} y}  =  \frac{1}{ \lambda b \mathbb{E}(e^{\gamma X})}.
 \end{align}
 Now set $\widetilde{\mathbb P}(X>x)= \widetilde{\ell}(x) x^{a+1-b} e^{-\lambda x^b}$, we know that $\widetilde{\ell}$ is a slowly varying function and this completes the proof of the case $b\in (0,1)$.
 
  The check for $\widetilde{\mathbb{E}}(|X|^k)<\infty$ is obvious since the right tail of $(X, \widetilde{\mathbb{P}})$ is stretched exponential and by \eqref{Change-of-measure}, for any $x>0$, 
  \begin{align}\label{left-tail-X}
  	\widetilde{\mathbb{P}}(X<-x) = \frac{1}{\mathbb{E}(e^{\gamma X})}\mathbb{E}(e^{\gamma X}1_{\{X<-x\}}) \leq  e^{-\gamma x}  \frac{1}{\mathbb{E}(e^{\gamma X})} \lesssim e^{-\gamma x}.
  \end{align}

 \hfill$\Box$

	\begin{prop}\label{prop2}
	Let $\beta \in (0,1)$ be fixed.
	\begin{itemize}
		\item[(i)] For any $x>K_0$ and $n\in\mathbb{N}$,
		\[
	      \beta^{n/2} \widetilde{\mathbb{E}}\left(e^{-\gamma (S_n-x)}1_{\{T_x^+=n\}}\right)  \lesssim_\beta   \ell(x)x^a e^{-\lambda x^b}.
		\]
		\item[(ii)] For each $n\in \mathbb{N}$,
		\[
		\lim_{x\to+\infty} \frac{ e^{\lambda x^b}}{ \ell(x)x^a } \widetilde{\mathbb{E}}\left(e^{-\gamma (S_n-x)}1_{\{T_x^+=n\}}\right)= \frac{1}{\gamma \mathbb{E}(e^{\gamma X})} \left(1+ \sum_{j=1}^{n-1}\widetilde{\mathbb{E}}\left( \left(1-e^{-\gamma \min_{1\leq i\leq j} S_i}\right)_+\right) \right),
		\]
	\end{itemize}
	here $x_+:= \max\{0,x\}$ and we adapt the convention $\sum_{j=1}^0 = 0$.
\end{prop}

\noindent
\textbf{Proof of Proposition \ref{prop2}  for $b=0$:}  
Since $a<-2$ by {\bf(H2)}, there exists a $\varepsilon\in (1/2,1)$ very close to $1$ such that $(-a)>\frac{2\varepsilon}{2\varepsilon -1}$, which is also equivalent to $(2\varepsilon-1)a+2\varepsilon<0$.   Also, the proof for $n=1$ follows by standard calculation, so we consider the case $n\geq 2$. 

\noindent
{\bf(Step 1)}. In this step, we define a good event $G_n$ and prove a uniform convergence result (see \eqref{e15} below). 
Noticing that $e^{-\gamma (S_n-x)}\leq e^{-\gamma x^{\varepsilon}}$ on the event $\{S_n>x+x^\varepsilon\}$, we have 
\begin{align}\label{e15-1}
	\sup_{n\in \mathbb{N}} \frac{  \beta^{n/2}  }{\ell(x) x^a} \widetilde{\mathbb{E}}\left(e^{-\gamma (S_n-x)}1_{\{T_x^+=n, S_n> x+ x^{\varepsilon}\}}\right) & \leq 	\frac{  e^{-\gamma x^{\varepsilon}} }{\ell(x) x^a} \stackrel{x\to+\infty}{\longrightarrow}0.  
\end{align}
Now define 
\begin{align}\label{def-of-F}
	F_n:= \left\{ \mbox{there exists at most one }i \leq n \mbox{ such that } X_i>x^\varepsilon \right\}.
\end{align}
Then by Lemma \ref{lem1},
\begin{align}\label{e15-2}
	& \frac{  \beta^{n/2}  }{\ell(x) x^a} \widetilde{\mathbb{E}}\left(e^{-\gamma (S_n-x)}1_{\{T_x^+=n, F_n^c\}}\right)  \leq 	\frac{  \beta^{n/2}  }{\ell(x) x^a}  \widetilde{\mathbb{P}}\left(  F_n^c \right)  \leq  \frac{\beta^{n/2} n^2 }{\ell(x) x^a}\widetilde{\mathbb{P}}\left( X>x^\varepsilon \right)^2 \nonumber\\
	& \lesssim_\beta \frac{\ell^2(x^{\varepsilon}) x^{2\varepsilon(a+1)}}{\ell(x) x^a}  = \frac{\ell^2(x^{\varepsilon}) }{\ell(x) } x^{(2\varepsilon-1)a + 2\varepsilon}.
\end{align}
Since $\frac{\ell^2(x^{\varepsilon}) }{\ell(x) } $ is still a slowly varying function at $\infty$ and $(2\varepsilon-1)a + 2\varepsilon<0$, the right hand side of the above display converges to $0$ as $x\to+\infty$ uniformly for all $n\in \mathbb{N}$. 
It follows from \eqref{left-tail-X} that
\begin{align}\label{e15-4}
	&\widetilde{\mathbb{E}}\left(e^{-\gamma (S_n-x)}1_{\{T_x^+=n, \min_{1\leq i\leq n} X_i <-x^{\varepsilon}\}}\right)  \leq  \widetilde{\mathbb{P}}\left( \min_{1\leq i\leq n} X_i< -x^{\varepsilon}\right)\nonumber\\
	& \leq n \widetilde{\mathbb{P}}\left(X<-x^{\varepsilon}\right) \lesssim  ne^{-\gamma x^{\varepsilon}} .
\end{align}
Therefore, we have 
\begin{align}\label{e15-3}
	& \sup_{n\in \mathbb{N}}\frac{ \beta^{n/2} }{\ell(x) x^a}\widetilde{\mathbb{E}}\left(e^{-\gamma (S_n-x)}1_{\{T_x^+=n, \min_{1\leq i\leq n} X_i < -x^{\varepsilon}\}}\right)  \nonumber\\
	&\lesssim  \sup_{n\in \mathbb{N}} n \beta^{n/2}  \frac{e^{-\gamma x^{\varepsilon}}}{\ell(x)x^a} \lesssim_\beta \frac{e^{-\gamma x^{\varepsilon}}}{\ell(x)x^a}\stackrel{x\to+\infty}{\longrightarrow}0. 
\end{align}
In conclusion,  define
\begin{align}
	G_n:= F_n\cap \{S_n\leq x+x^\varepsilon\} \cap \Big\{\min_{1\leq i\leq n} X_i \geq -x^{\varepsilon} \Big\},
\end{align} 
then combining \eqref{e15-1}, \eqref{e15-2} and \eqref{e15-3}, we have 
\begin{align}\label{e15}
	\lim_{x\to+\infty} \sup_{n\in \mathbb{N}}\frac{ \beta^{n/2} }{\ell(x) x^a} \widetilde{\mathbb{E}}\left(e^{-\gamma (S_n-x)}1_{\{T_x^+=n, G_n^c\}}\right) =0. 
\end{align}

\noindent
{\bf(Step 2)}. In this step, we  restrict the expectarion in $G_n$ and prove the proposition. 
 For simplicity we only consider large $x$ and  $n\geq 2$ since   $\inf_{x\in (K_1, K)}\ell(x)x^a >0$ for all $K>K_1$ by \eqref{Def-of-K-0} and the case $n=1$ follows by direct calculation. 
 
 \noindent
 {\bf Case 1: $n  \geq (\log x)^2$.}  We have (noticing that this part is only used in the proof of (i)),
\begin{align}\label{e17}
	 \beta^{n/2} \widetilde{\mathbb{E}}\left(e^{-\gamma (S_n-x)}1_{\{T_x^+=n, G_n\}}\right) \leq  \beta^{n/2}\leq  \beta^{(\log x)^2/2}\lesssim_\beta   \ell(x)x^a.
\end{align}

\noindent
{\bf Case 2: $n  <(\log x)^2$.}  Since on $\{T_x^+=n, \max_{1\leq i\leq n} X_i \leq x^{\varepsilon}\}$, we have  $S_n \leq n x^{\varepsilon}\leq (\log x)^2 x^{\varepsilon} <x$ for large $x$. Therefore, $\{T_x^+=n, \max_{1\leq i\leq n} X_i \leq x^{\varepsilon}\}\cap G_n=\emptyset$, which implies that for large $x$, 
\begin{align}\label{e13}
		  & \widetilde{\mathbb{E}}\left(e^{-\gamma (S_n-x)}1_{\{T_x^+=n, G_n\}}\right)  \nonumber\\
		  &= 	 \widetilde{\mathbb{E}}\left(e^{-\gamma (S_n-x)}1_{\{T_x^+=n, S_n\leq x+x^\varepsilon,  \exists   j\leq n, s.t. X_j>x^\varepsilon,\ \max_{i\neq j} |X_i| \leq x^{\varepsilon}  \}}\right) \nonumber\\
		  & = \sum_{j=1}^n  \widetilde{\mathbb{E}}\left(e^{-\gamma (S_n-x)}1_{\{T_x^+=n, S_n\leq x+x^\varepsilon,   X_j>x^\varepsilon,\ \max_{i\neq j} |X_i| \leq x^{\varepsilon}  \}}\right) .
\end{align}
For each $j\leq n$, $\{S_n>x\}\cap \{\max_{i\neq j} |X_i| \leq x^{\varepsilon}   \}\subset \{X_j \geq x- nx^\varepsilon\}\subset \{X_j >x^\varepsilon\}$. Thus,
\begin{align}\label{e16}
	  & \widetilde{\mathbb{E}}\left(e^{-\gamma (S_n-x)}1_{\{T_x^+=n, G_n\}}\right)   = \sum_{j=1}^n  \widetilde{\mathbb{E}}\left(e^{-\gamma (S_n-x)}1_{\{T_x^+=n, S_n\leq x+x^\varepsilon, \ \max_{i\neq j} |X_i| \leq x^{\varepsilon}  \}}\right) .
\end{align}
To prove (i),   we have 
\begin{align}
	  & \widetilde{\mathbb{E}}\left(e^{-\gamma (S_n-x)}1_{\{T_x^+=n, G_n\}}\right)   \leq \sum_{j=1}^n  \widetilde{\mathbb{E}}\left(e^{-\gamma (S_n-x)}1_{\{ x<S_n\leq x+x^\varepsilon, \ \max_{i\neq j} |X_i| \leq x^{\varepsilon}  \}}\right) \nonumber\\
	  & = n \widetilde{\mathbb{E}}\left(1_{\{ \max_{1\leq i\leq n-1} |X_i| \leq x^{\varepsilon}  \}} \widetilde{\mathbb{E}} \left(e^{-\gamma (X-y)}1_{\{ y< X\leq y+x^\varepsilon\}}\right)\Big|_{y=x-S_{n-1}}\right). 
\end{align}
Since $|S_{n-1}|\leq nx^\varepsilon \leq (\log x)^2 x^{\varepsilon} $ on $\{ \max_{1\leq i\leq n-1} |X_i| \leq x^{\varepsilon}  \}$, we have
\begin{align}\label{e18}
	& \frac{ \beta^{n/2}}{\ell(x) x^a} \widetilde{\mathbb{E}}\left(e^{-\gamma (S_n-x)}1_{\{T_x^+=n, G_n\}}\right) \leq 	 \frac{ n\beta^{n/2} }{\ell(x) x^a} \sup_{|y-x|\leq (\log x)^2 x^{\varepsilon}} \widetilde{\mathbb{E}} \left(e^{-\gamma (X-y)}1_{\{ y< X\leq y+x^\varepsilon\}}\right)\nonumber\\
	& \lesssim_\beta  \frac{  1}{\ell(x) x^a}\sup_{|y-x|\leq (\log x)^2 x^{\varepsilon}}  \int_0^{x^\varepsilon} \ell(z+y) (z+y)^a e^{-\gamma z}\mathrm{d}z \nonumber\\
	& \lesssim 1. 
\end{align}
Combining \eqref{e15}, \eqref{e17} and \eqref{e18}, we complete the proof of (i).

Now we are ready to prove (ii). Since $\max_{1\leq k\leq j-1}S_k \leq n x^\varepsilon <x$ for large $x$ on $\{\max_{i\neq j} |X_i| \leq x^{\varepsilon}\}$, we have for each $n\in \mathbb{N}$ and $1\leq j\leq n-1$, 
\begin{align}\label{e19}
	& \widetilde{\mathbb{E}}\left(e^{-\gamma (S_n-x)}1_{\{T_x^+=n, S_n\leq x+x^\varepsilon, \ \max_{i\neq j} |X_i| \leq x^{\varepsilon}  \}}\right) \nonumber\\
	& =  \widetilde{\mathbb{E}}\left(e^{-\gamma (S_n-x)}1_{\{ X_j \leq x-\max_{j+1\leq k\leq n}(S_{k-1}-X_j),\  X_j + (S_n-X_j) \in (x,x+x^\varepsilon], \ \max_{i\neq j} |X_i| \leq x^{\varepsilon}  \}}\right). 
\end{align}
If $j=n$, we  remove the first event $\{ X_j \leq x-\max_{j+1\leq k\leq n}(S_{k-1}-X_j)\}$. 
 Set $Y=x-( S_n-X_j)$, $U=S_n- \max_{j\leq k\leq n-1} S_k$ for $1\leq j \leq n-1$ and $U=\infty$ for $j=n$. According to the independence between $X_j$ and $(U, Y)$, we have for all $1\leq j\leq n$, 
 \begin{align}
 		& \widetilde{\mathbb{E}}\left(e^{-\gamma (S_n-x)}1_{\{T_x^+=n, S_n\leq x+x^\varepsilon, \ \max_{i\neq j} |X_i| \leq x^{\varepsilon}  \}}\right) \nonumber\\
 	& =  \widetilde{\mathbb{E}}\left(e^{-\gamma (X_j-Y)} 1_{\{ X_j \leq U+Y,\  X_j  \in (Y,Y+x^\varepsilon], \ \max_{i\neq j} |X_i| \leq x^{\varepsilon} , U>0 \}}\right)\nonumber\\
 	& = \widetilde{\mathbb{E}}\left(1_{\{\max_{i\neq j} |X_i| \leq x^{\varepsilon} \}} 1_{\{U>0\}}\widetilde{\mathbb{E}}\left(e^{-\gamma (X_j-y)}1_{\{ X_j \leq u+y,\  X_j  \in (y,y+x^\varepsilon] \}}\right)\Big|_{y=Y, u=U}\right).
 \end{align}
We also mention here that on the event $\{\max_{i\neq j} |X_i| \leq x^{\varepsilon} \}$,  $ |Y-x| \leq n x^\varepsilon \leq (\log x)^2 x^\varepsilon$.  Since  the law of $U= S_n- \max_{j\leq k\leq n-1} S_k$ is independent of $x$ and that  uniformly for all $|y-x| \leq (\log x)^2 x^\varepsilon$,
\begin{align}
	 & \widetilde{\mathbb{E}}\left(e^{-\gamma (X_j-y)}1_{\{ X_j \leq u+y,\  X_j  \in (y,y+x^\varepsilon] \}}\right)  = \frac{1}{\mathbb{E}(e^{\gamma X})}\int_0^{\min\{x^{\varepsilon}, u\} } \ell(z+y) (z+y)^a e^{-\lambda z} \mathrm{d}z,
\end{align}
 taking $x\to+\infty$ in the above limit implies that uniformly on $|y-x| \leq (\log x)^2 x^\varepsilon$,
\begin{align}
	& \frac{ 1 }{\ell(x)x^a}\widetilde{\mathbb{E}}\left(e^{-\gamma (X_j-y)}1_{\{ X_j \leq u+y,\  X_j  \in (y,y+x^\varepsilon] \}}\right)= \frac{1}{\mathbb{E}(e^{\gamma X})}\int_0^{\min\{ x^{\varepsilon}, u\} }\frac{ \ell(z+y)}{\ell (x)} \left(\frac{z+y}{x}\right)^a e^{-\lambda z} \mathrm{d}z \nonumber\\
	 & \stackrel{x\to+\infty}{\longrightarrow} \frac{1}{\mathbb{E}(e^{\gamma X})} \int_0^ue^{-\lambda z}\mathrm{d}z.
\end{align}
Noticing that $S_n- \max_{j\leq k\leq n-1} S_k$ is equal in law to $\min_{1\leq i\leq n-j} S_i$, we conclude that 
\begin{align}
	& \lim_{x\to+\infty} \frac{1}{\ell(x)x^a}\widetilde{\mathbb{E}}\left(e^{-\gamma (S_n-x)}1_{\{T_x^+=n, G_n\}}\right)   \nonumber\\
	&= \frac{1}{\mathbb{E}(e^{\gamma X})} \left(\frac{1}{\gamma}+ \sum_{j=1}^{n-1}\widetilde{\mathbb{E}}\left(1_{\{\min_{1\leq i\leq j} S_i>0 \}}\int_0^{\min_{1\leq i\leq j} S_i} e^{-\gamma z}\mathrm{d} z\right) \right).
\end{align}
Therefore, (ii) follows from \eqref{e15} and the above display. We are done. 

\hfill$\Box$

\noindent
\textbf{Proof of Proposition \ref{prop2}  for $b\in (0,1)$:}  The proof is similar to the case $b=0$ with more delicate details. Since the proof for $n=1$ follows by standard calculation,   we consider the case $n\geq 2$.  Recall that $\varepsilon_0= \frac{1}{2}\min\{ b, 1-b\}$. 

\noindent
{\bf(Step 1)}.  In this step, we define a good event $I_n$ and show that the expectation on $I_n^c$ is negligible (see \eqref{e19-4} below).
Combining Proposition \ref{prop1} (i) and Lemma \ref{lem1},  for all $x>K_0$ and $n\in \mathbb{N}$,
\begin{align}\label{e19-1}
		& \frac{  \beta^{n/2} e^{\lambda x^b}}{\ell(x) x^a } \widetilde{\mathbb{E}}\left(e^{-\gamma (S_n-x)}1_{\{T_x^+=n, S_n> x+ x^{\varepsilon_0}\}}\right)  \leq 	\frac{  \beta^{n/2} e^{\lambda x^b-\gamma x^{\varepsilon_0}}}{\ell(x) x^a   }\widetilde{\mathbb{P}}\left(T_x^+=n\right)  \nonumber\\
		&\lesssim_\beta 	\frac{ e^{\lambda x^b-\gamma x^{\varepsilon_0}} }{\ell(x) x^a } \widetilde{\mathbb{P}}\left(X>x\right) \lesssim e^{-\gamma x^{\varepsilon_0} }x^{1-b}.
\end{align}
 Define 
\begin{align}\label{def-of-H}
	H_n:= \{\mbox{there exists at most one }i \leq n \mbox{ such that } X_i>x-2x^{\varepsilon_0}  \}.
\end{align}
Combining  \eqref{e9-2} (with $m$ replaced by $\beta$) and Lemma \ref{lem1}, we have for all $x>K_0$ and $n\in \mathbb{N}$, 
\begin{align}\label{e19-2}
		& \frac{  \beta^{n/2}e^{\lambda x^b} }{\ell(x) x^a } \widetilde{\mathbb{E}}\left(e^{-\gamma (S_n-x)}1_{\{T_x^+=n, H_n^c\}}\right)  \leq 	\frac{  \beta^{n/2} e^{\lambda x^b}}{\ell(x) x^a } \widetilde{\mathbb{P}}\left(H_n^c\right)  \lesssim_\beta \frac{  e^{\lambda x^b} }{\ell(x) x^a } \widetilde{\mathbb{P}}\left(X>x\right)^2 \nonumber\\
		& = \frac{   \widetilde{\ell}^2(x) x^{2(a+1-b)} e^{-\lambda x^b}}{\ell(x) x^a  } \lesssim \ell(x) x^{a+2(1-b)} e^{-\lambda x^b} \nonumber\\
		&\lesssim e^{-\lambda x^b /2}, 
\end{align}
where in the last inequality we inequality $\ell(x) x^{a+2(1-b)+1}\lesssim e^{\lambda x^b /2}$ for all $x>K_0$. 
Define 
\begin{align}\label{def-of-I-n}
		I_n:= H_n\cap \{S_n\leq x+x^{\varepsilon_0}\} ,
\end{align}
then it follows from \eqref{e19-1} and  \eqref{e19-2}  that 
\begin{align}\label{e19-4}
	\lim_{x\to+\infty} \sup_{n\in \mathbb{N}}\frac{ \beta^{n/2}e^{\lambda x^b} }{\ell(x) x^a } \widetilde{\mathbb{E}}\left(e^{-\gamma (S_n-x)}1_{\{T_x^+=n, I_n^c\}}\right)=0. 
\end{align}

\noindent
{\bf(Step 2)}.  In this step, we restrict the expectation in $I_n$ for   $n\geq (\log x)^{5/2}$.

\noindent
{\bf Case 1: $n\geq x^{(1+b)/2}$.}
In this case,  $n\log (1/\beta) /2 >2\lambda x^b$ for large $x$, which implies that
\begin{align}\label{e20-1}
	\frac{ \beta^{n/2} e^{\lambda x^b}}{\ell(x) x^a }\widetilde{\mathbb{E}}\left(e^{-\gamma (S_n-x)}1_{\{T_x^+=n, I_n\}}\right) \leq 	\frac{ \beta^{n/2} e^{\lambda x^b}}{\ell(x) x^a } \leq \frac{ e^{-\lambda x^b}}{\ell(x) x^a } \lesssim 1. 
\end{align}

\noindent
{\bf Case 2: $(\log x)^{5/2} \leq n < x^{(1+b)/2}$.} It follows from \eqref{e8-1} that 
\begin{align}
		& \frac{e^{\lambda x^b} }{\ell(x) x^a }   \widetilde{\mathbb{E}}\left(e^{-\gamma (S_n-x)}1_{\{T_x^+=n, I_n,\max_{1\leq i\leq n} X_i \leq x-2x^{\varepsilon_0}\}}\right) \nonumber\\
		&\leq \frac{e^{\lambda x^b} }{\ell(x) x^a }   \widetilde{\mathbb{P}}\left( S_n >x, \max_{1\leq i\leq n} X_i \leq x-2x^{\varepsilon_0} \right)\nonumber\\
		& \lesssim   \frac{e^{(\log x)^2+\left(C_*\theta^2(x)+\widetilde{\mu} \theta(x) \right)n} }{\ell(x) x^a }  \nonumber\\
		&\lesssim  e^{\left(n^{-1/5}+ (|a|+1)n^{-3/5}+ C_*\theta^2(x) +\widetilde{\mu} \theta(x) \right)n}  ,
\end{align}
where $\widetilde{\mu}:= \widetilde{\mathbb{E}}(X)$ and the last inequality follows from the fact that $\ell (x) x^{a}\gtrsim x^{-|a|-1}$ and $\log x < n^{2/5}$.  Taking $x$ sufficiently large such that 
$n^{-1/5}+ (|a|+1)n^{-3/5}+ C_*\theta^2(x)+\widetilde{\mu} \theta(x)\leq -\frac{1}{3}\log \beta$ for all $n\geq  (\log x)^{5/2}$,  we deduce that 
\begin{align}
		& \frac{ \beta^{n/2}e^{\lambda x^b} }{\ell(x) x^a }  \widetilde{\mathbb{E}}\left(e^{-\gamma (S_n-x)}1_{\{T_x^+=n, I_n,\max_{1\leq i\leq n} X_i \leq x-2x^{\varepsilon_0}\}}\right) \nonumber\\
		& \lesssim \beta^{n/2} x^{1-b}e^{ -n (\log \beta)/3} =\beta^{n/6}x^{1-b}  \leq \beta^{(\log x)^{5/2}/6}x^{1-b} \nonumber\\
		& \lesssim_\beta 1.
\end{align}
Therefore,  for large $x$ and $(\log x)^{5/2} \leq n < x^{(1+b)/2}$, it holds that
\begin{align}
	& \frac{ \beta^{n/2} e^{\lambda x^b}}{\ell(x) x^a }  \widetilde{\mathbb{E}}\left(e^{-\gamma (S_n-x)}1_{\{T_x^+=n, I_n\}}\right) \nonumber\\
	& \lesssim_\beta 1+  \frac{ \beta^{n/2} e^{\lambda x^b}}{\ell(x) x^a } \widetilde{\mathbb{E}}\left(e^{-\gamma (S_n-x)}1_{\{T_x^+=n, I_n, \ \exists !\ 1\leq  j\leq n, X_j >x-2x^{\varepsilon_0}\}}\right) \nonumber\\
	& \leq 1+  \frac{ n \beta^{n/2} e^{\lambda x^b}}{\ell(x) x^a }\widetilde{\mathbb{E}}\left(e^{-\gamma (S_n-x)}1_{\{S_n\in (x, x+x^{\varepsilon_0}), X_n >x-2x^{\varepsilon_0}  \}}\right).
\end{align}
Since $\{X_n> x-2x^{\varepsilon_0}\}\cap  \{S_n<x+x^{\varepsilon_0}\}\subset \{S_{n-1} <3x^{\varepsilon_0}\}$, we obtain that 
\begin{align}\label{e21}
		& \frac{ \beta^{n/2} e^{\lambda x^b}}{\ell(x) x^a }  \widetilde{\mathbb{E}}\left(e^{-\gamma (S_n-x)}1_{\{T_x^+=n, I_n\}}\right)  \nonumber\\
	& \lesssim_\beta  1+  \frac{ n \beta^{n/2} e^{\lambda x^b}}{\ell(x) x^a } \sup_{y>x-3x^{\varepsilon_0}} \widetilde{\mathbb{E}}\left(e^{-\gamma (X-y)}1_{\{X\in (y, y+x^{\varepsilon_0}) \}}\right) \nonumber\\
	& \lesssim_\beta 1+  \frac{ e^{\lambda x^b}}{ \ell(x) x^a } \sup_{y>x-3x^{\varepsilon_0}}  \int_0^{3x^{\varepsilon_0}}e^{-\gamma z} \ell(z+y)(z+y)^a e^{-\lambda (z+y)^b} \mathrm{d}z. 
\end{align}
Using the inequality $\ell(z+y) \lesssim \ell(x)$ for any $z<3x^{\varepsilon_0}, x-3x^{\varepsilon_0}<y<2x$ and $\ell(z+y)\lesssim z+y \lesssim y$ for any $z<3x^{\varepsilon_0}, y\geq 2x$, we have 
\begin{align}\label{e22-2}
	&  \frac{ e^{\lambda x^b}}{ \ell(x) x^a }  \sup_{y>2x} \int_0^{3x^{\varepsilon_0}} e^{-\gamma z} \ell(z+y) (z+y)^{a} e^{-\lambda (z+y)^b} \mathrm{d}z\nonumber\\
	& \lesssim  \frac{ e^{\lambda x^b}}{ \ell(x) x^a } \sup_{y>2x} \int_0^{3x^{\varepsilon_0}} e^{-\gamma z}   y^{|a|+1} e^{-\lambda y^b} \mathrm{d}z\nonumber\\
	&\lesssim  \frac{ e^{\lambda x^b}}{ \ell(x) x^a }  \sup_{y>2x}    y^{|a|+1} e^{-\lambda y^b}  \lesssim \frac{e^{\lambda x^b}}{\ell(x) x^a }   (2x)^{|a|+1} e^{-\lambda (2x)^b}  \lesssim 1
\end{align}
and 
\begin{align}\label{e22-3}
	&  \frac{ e^{\lambda x^b}}{ \ell(x) x^a } \sup_{x-3x^{\varepsilon_0}<y<2x} \int_0^{3x^{\varepsilon_0}} e^{-\gamma z} \ell(z+y) (z+y)^{a} e^{-\lambda (z+y)^b} \mathrm{d}z\nonumber\\
	&\lesssim  \frac{ e^{\lambda x^b}}{ \ell(x) x^a } \sup_{x-3x^{\varepsilon_0}<y<2x} \int_0^{3x^{\varepsilon_0}} e^{-\gamma z} \ell(x) x^{a} e^{-\lambda (z+y)^b} \mathrm{d}z \nonumber\\
	& \leq e^{\lambda x^b-\lambda (x-3x^{\varepsilon_0})^b}    \int_0^{\infty} e^{-\gamma z} \mathrm{d}z \lesssim 1,
\end{align}
where in the last inequality we used the fact that $\lim_{x\to+\infty} ((x-3x^{\varepsilon_0})^b-x^b)=0$.
Combining \eqref{e21}, \eqref{e22-2} and \eqref{e22-3}, we obtain that 
\begin{align}\label{e20-2}
		& \frac{ \beta^{n/2} e^{\lambda x^b}}{\ell(x) x^a }   \widetilde{\mathbb{E}}\left(e^{-\gamma (S_n-x)}1_{\{T_x^+=n, I_n\}}\right) \lesssim_\beta 1.
\end{align}
In conclusion, combining \eqref{e20-1} and \eqref{e20-2}, we get that 
\begin{align}\label{e20}
	\sup_{n \in \mathbb{N}, n \geq  (\log x)^{5/2}} \frac{ \beta^{n/2} e^{\lambda x^b}}{\ell(x) x^a }   \widetilde{\mathbb{E}}\left(e^{-\gamma (S_n-x)}1_{\{T_x^+=n, I_n\}}\right)   \lesssim_\beta  1.
\end{align}

\noindent
{\bf(Step 3)}. In this step, we restrict the expectation in $I_n$ for   $n<(\log x)^{5/2}$ and prove   (i). 
Noticing that $\{S_n >x\}$ implies the existence of some $1\leq j\leq n$ such that $X_j> x/n > x/(\log x)^{5/2}$. Therefore, when $x$ is large enough, we have $\{T_x^+=n\}\cap I_n \subset I_n\cap \{\exists\ 1\leq j \leq n,\ X_j \geq 2x^{\varepsilon_0}\} =  \{\exists\  1\leq j \leq n,\ 2x^{\varepsilon_0}\leq X_j \leq x-2x^{\varepsilon_0}\} \cup \{\exists\ 1\leq j \leq n,\  X_j > x-2x^{\varepsilon_0},\ \max_{i\neq j}X_i< 2x^{\varepsilon_0}\}  $. Based on this observation, we have the following decomposition: 
\begin{align}\label{decomposition-I-n}
	& \widetilde{\mathbb{E}}\left(e^{-\gamma (S_n-x)}1_{\{T_x^+=n, I_n\}}\right)\nonumber\\
	&= \widetilde{\mathbb{E}}\left(e^{-\gamma (S_n-x)}1_{\{T_x^+=n, I_n, \exists\ 1\leq j \leq n,\ 2x^{\varepsilon_0}\leq X_j \leq x-2 x^{\varepsilon_0} \}}\right) \nonumber\\
	&\qquad + \widetilde{\mathbb{E}}\left(e^{-\gamma (S_n-x)}1_{\{T_x^+=n, I_n, \exists \ 1\leq j \leq n,\  X_j > x-2 x^{\varepsilon_0}, \max_{i\neq j} X_i <2x^{\varepsilon_0} \}}\right) \nonumber\\
	&=: A_1+A_2. 
\end{align}

\noindent
{\bf Estimate for $A_1$.}
 Noticing that for $1\leq j\leq n$, $y=x-X_j\geq 2x^{\varepsilon_0}$ on the event $\{X_j\leq x-2x^{\varepsilon_0}\}$, we may take $x$ sufficiently large such that $(\log x)^{5/2}\leq (\log (2x^{\varepsilon_0}))^3$. Therefore, by  \eqref{e6}, when $x$ is large enough, for all $n< (\log x)^{5/2}$ and $y\geq 2x^{\varepsilon_0}$, $\widetilde{\mathbb{P}}(S_{n}-X_j>y) \leq 2 (n-1) \widetilde{\mathbb{P}}(X>y)$, which implies that  
\begin{align}
	& \frac{ e^{\lambda x^b}}{\ell(x) x^a }  A_1 \leq \frac{e^{\lambda x^b}}{\ell(x) x^a }  \sum_{j=1}^n \widetilde{\mathbb{P}}\left(S_{n}-X_j> x- X_j, 2x^{\varepsilon_0}\leq  X_j \leq x-2 x^{\varepsilon_0} \right) \nonumber\\
	& = \frac{ ne^{\lambda x^b}}{\ell(x) x^a } \widetilde{\mathbb{E}}\left(1_{\{ 2x^{\varepsilon_0}\leq X_n \leq x-2 x^{\varepsilon_0}\}} \widetilde{\mathbb{P}} (S_{n-1}\geq y)|_{y=x-X_n} \right)\nonumber\\
	& \leq \frac{ 2n(n-1) e^{\lambda x^b}}{\ell(x) x^a } \widetilde{\mathbb{E}}\left(1_{\{ 2x^{\varepsilon_0}\leq  X_n \leq x-2 x^{\varepsilon_0}\}} \widetilde{\mathbb{P}} (X\geq y)|_{y=x-X_n} \right).
\end{align}
 Combining   Lemma \ref{lem1} and \eqref{e10-1} (with $|a|+1$ replaced by $|a|+2-b$) and the fact that $\ell(x) x^{a}\gtrsim x^{-|a|-1}$, we deduce that 
\begin{align}\label{e21-4}
	& \sup_{n< (\log x)^{5/2}}\frac{ \beta^{n/2} e^{\lambda x^b}}{\ell(x) x^a } A_1 \nonumber\\
	&\lesssim_\beta \sup_{n< (\log x)^{5/2}} \frac{ e^{\lambda x^b}}{\ell(x) x^a } \widetilde{\mathbb{E}}\left(1_{\{ 2x^{\varepsilon_0}\leq  X_n \leq x-2 x^{\varepsilon_0}\}} |x-X_n|^{|a|+2-b} e^{-\lambda (x-X_n)^{b}} \right)\nonumber\\
	&\lesssim \frac{  x^{2(|a|+2-b) +1} e^{-\lambda (2x^{\varepsilon_0})^{b}} }{\ell(x) x^a } \lesssim   x^{3|a|+6-2b}   e^{-\lambda (2x^{\varepsilon_0})^{b}}\stackrel{x\to+\infty}{\longrightarrow}0.
\end{align}

\noindent
{\bf Estimate for $A_2$.}
We have the upper bound
\begin{align} 
	& \frac{  e^{\lambda x^b}}{\ell(x) x^a } A_2  \leq  \frac{n e^{\lambda x^b}}{\ell(x) x^a } \widetilde{\mathbb{E}}\left(e^{-\gamma (S_n-x)}1_{\{S_n\in (x, x+x^{\varepsilon_0}),   X_n > x-2 x^{\varepsilon_0}\}}\right) .
\end{align}
Since $\{X_n>x-2x^{\varepsilon_0}\}\cap \{S_n < x+x^{\varepsilon_0}\}\subset \{S_{n-1}<3x^{\varepsilon_0}\}$, we deduce that 
\begin{align}\label{e21-3}
	& \frac{ \beta^{n/2} e^{\lambda x^b}}{\ell(x) x^a }  A_2 \leq  \frac{n \beta^{n/2}e^{\lambda x^b} }{\ell(x) x^a }\widetilde{\mathbb{E}}\left(1_{\{S_{n-1}<3x^{\varepsilon_0}\}} \widetilde{\mathbb{E}} \left(e^{-\gamma (X-y)}1_{\{X\in (y, y+x^{\varepsilon_0})  \}}\right)\big|_{y=x-X_n}\right)  \nonumber\\
	& \leq  \frac{n \beta^{n/2} e^{\lambda x^b}}{\ell(x) x^a }   \sup_{y>x-3x^{\varepsilon_0}} \int_0^{3x^{\varepsilon_0}} e^{-\gamma z} \ell(z+y) (z+y)^{a} e^{-\lambda (z+y)^b} \mathrm{d}z\nonumber\\
	&\lesssim_\beta 1,
\end{align}
where in the last inequality we used \eqref{e22-2} and \eqref{e22-3} . 
Combining \eqref{decomposition-I-n}, \eqref{e21-4} and \eqref{e21-3}, we get that 
\begin{align}\label{e21-5}
	\sup_{n\in \mathbb{N}, n< (\log x)^{5/2}}\frac{ \beta^{n/2}e^{\lambda x^b} }{\ell(x) x^a } \widetilde{\mathbb{E}}\left(e^{-\gamma (S_n-x)}1_{\{T_x^+=n, I_n\}}\right) \lesssim_\beta 1.
\end{align}
Now (i) follows from \eqref{e19-4}, \eqref{e20} and \eqref{e21-5}. 

\noindent
{\bf(Step 4)}. In this step, we prove (ii). For each fixed $n\in \mathbb{N}$, we assume that $x$ is large enough such that $n<(\log x)^{5/2}$.  Combining \eqref{e19-4}, \eqref{decomposition-I-n} and \eqref{e21-4}, we have 
\begin{align}\label{e22}
	& \lim_{x\to+\infty} \frac{  e^{\lambda x^b} }{\ell(x) x^a } \widetilde{\mathbb{E}}\left(e^{-\gamma (S_n-x)}1_{\{T_x^+=n\}}\right) = \lim_{x\to+\infty} \frac{  e^{\lambda x^b} }{\ell(x) x^a } \widetilde{\mathbb{E}}\left(e^{-\gamma (S_n-x)}1_{\{T_x^+=n, I_n \}}\right) \nonumber\\
	& =  \lim_{x\to+\infty} \frac{  e^{\lambda x^b}}{\ell(x) x^a } (A_1+ A_2)=  \lim_{x\to+\infty} \frac{  e^{\lambda x^b} }{\ell(x) x^a } A_2  \nonumber\\
	& = \sum_{j=1}^n  \lim_{x\to+\infty} \frac{ e^{\lambda x^b}}{\ell(x) x^a } \widetilde{\mathbb{E}}\left(e^{-\gamma (S_n-x)}1_{\{T_x^+=n, I_n , \  X_j>x-2x^{\varepsilon_0},\ \max_{i\neq j} X_i< 2x^{\varepsilon_0}\}}\right).
\end{align}
The rest part of the proof is exactly the same as $b=0$. 
Since  $\widetilde{\mathbb{E}}(|X|^k)<\infty$ for all $k>1$ by Lemma \ref{lem1},  set $k_0:= 1/\varepsilon_0$, then $k_0\varepsilon_0 > 1-b$ and for each $1\leq j\leq n$, 
\begin{align}\label{e27-1}
	 &\limsup_{x\to+\infty}  \frac{ e^{\lambda x^b}}{\ell(x) x^a } \widetilde{\mathbb{E}}\left(e^{-\gamma (S_n-x)}1_{\{T_x^+=n, I_n , \  X_j>x-2x^{\varepsilon_0},\ |S_n- S_j|>x^{\varepsilon_0}\}}\right) \nonumber\\
	& \leq \lim_{x\to+\infty}  \frac{  e^{\lambda x^b} }{\ell(x) x^a } \widetilde{\mathbb{P}}\left( X_j>x-2x^{\varepsilon_0}\right) \widetilde{\mathbb{P}} \left(|S_n-S_j|>x^{\varepsilon_0}\right)\nonumber\\
	& \leq  \lim_{x\to+\infty}  \frac{  (n-1) e^{\lambda x^b}}{\ell(x) x^a } \widetilde{\mathbb{P}}\left( X>x-2x^{\varepsilon_0}\right) \widetilde{\mathbb{P}} \left(|X|>x^{\varepsilon_0}/n \right)  \nonumber\\
	& \leq  \frac{n-1}{\lambda b \mathbb{E}(e^{\gamma X})} \lim_{x\to+\infty} x^{1-b} \frac{\widetilde{\mathbb{E}}(|X|^{k_0})}{(x^{\varepsilon_0}/n)^{k_0}}=0.
\end{align}
Therefore, plugging this back to \eqref{e22} and noticing that $\{T_x^+=n\}\cap\{|S_n-X_j|\leq x^{\varepsilon_0}\}\subset \{X_j>x-2x^{\varepsilon_0}\}$,  we obtain
\begin{align}\label{e22-1}
	& \lim_{x\to+\infty} \frac{  e^{\lambda x^b} }{\ell(x) x^a } \widetilde{\mathbb{E}}\left(e^{-\gamma (S_n-x)}1_{\{T_x^+=n\}}\right)   \nonumber\\
	& =\sum_{j=1}^n  \lim_{x\to+\infty} \frac{  e^{\lambda x^b} }{\ell(x) x^a } \widetilde{\mathbb{E}}\left(e^{-\gamma (S_n-x)}1_{\{T_x^+=n, I_n , \  |S_n-X_j|\leq x^{\varepsilon_0},\ \max_{i\neq j} X_i< 2x^{\varepsilon_0}\}}\right).
\end{align}
Since $\max_{1\leq k\leq j-1} S_k\leq 2n x^{\varepsilon_0}< x$ on $\{\max_{i\neq j} X_i<2x^{\varepsilon_0}\}$ for large $x$, we have for $1\leq j\leq n$,
\begin{align}
	& \widetilde{\mathbb{E}}\left(e^{-\gamma (S_n-x)}1_{\{T_x^+=n, I_n , \  |S_n-X_j|\leq x^{\varepsilon_0},\ \max_{i\neq j} X_i< 2x^{\varepsilon_0}\}}\right)\nonumber\\
	& = \widetilde{\mathbb{E}}\left(e^{-\gamma (S_n-x)}1_{\{T_x^+=n, I_n , \  |S_n-X_j|\leq x^{\varepsilon_0},\ \max_{i\neq j} X_i< 2x^{\varepsilon_0}\}}\right)\nonumber\\
	& = \widetilde{\mathbb{E}}\left( 1_{\{ \max_{i\neq j} X_i< 2x^{\varepsilon_0}, |x-Y| \leq x^{\varepsilon_0}, U>0\}} \widetilde{\mathbb{E}}\left(e^{-\gamma (X_j-y)}1_{\{ X_j \leq u+y, X_j \in (y, y+x^\varepsilon_0)\}}\right)\big|_{y=Y, u=U}\right),
\end{align}
where $Y= x-(S_n-X_j)$, $U=S_n- \max_{j\leq k\leq n-1} S_k$ for $1\leq j\leq n-1$ and $U=\infty$ for $j=n$ and that the last equality follows from the independence between $(Y, U)$ and $X_j$.
Since 
\[
\ell(z+y)/\ell(x) \to 1,\quad  (z+y)/x \to 1 \quad \mbox{and}\quad (z+y)^b -x^b \to 0
\]
uniformly for $z\in (0, x^{\varepsilon_0})$ and $|y-x|\leq  x^{\varepsilon_0}$ as $x\to+\infty$, we conclude that 
\begin{align}
	& \lim_{x\to+\infty} \frac{ e^{\lambda x^b} }{\ell(x) x^a } \widetilde{\mathbb{E}}\left(e^{-\gamma (S_n-x)}1_{\{T_x^+=n, I_n , \  |S_n-X_j| \leq x^{\varepsilon_0},\ \max_{i\neq j} X_i< 2x^{\varepsilon_0}\}}\right) \nonumber\\
	&= \lim_{x\to+\infty} \frac{  \widetilde{\mathbb{E}}\left( 1_{\{ \max_{i\neq j} X_i< 2x^{\varepsilon_0},  |x-Y|\leq x^{\varepsilon_0}, U>0\}} \int_{0}^{ \min\{x^{\varepsilon_0}, U\} }  e^{-\gamma z} \ell(z+Y) (z+Y)^a e^{-\lambda (z+Y)^b}\mathrm{d}z\right)}{\ell(x) x^a e^{-\lambda x^b}} \nonumber\\
	& = \lim_{x\to+\infty}    \widetilde{\mathbb{E}}\left( 1_{\{ \max_{i\neq j} X_i< 2x^{\varepsilon_0},  |S_n-X_j|\leq x^{\varepsilon_0}, U>0\}}  \int_{0}^{ \min\{x^{\varepsilon_0}, U\} }  e^{-\gamma z}   \mathrm{d}z\right) \nonumber\\
	& = \widetilde{\mathbb{E}}\left( 1_{\{ U>0\}}  \int_{0}^{ U }  e^{-\gamma z}   \mathrm{d}z\right) .
\end{align}
Plugging this back to \eqref{e22-1} completes the proof of (ii).

\hfill$\Box$

We  also need the following important proposition in the proof of Theorem \ref{thm2}. 

\begin{prop}\label{prop3}
	Assume {\bf(H2)}.  Let $\delta, \beta \in (0,1)$ be fixed.
	
	(i) For all $k\in \mathbb{N}$ with $k\delta\geq K_0$ and $x>K_0$, 
	\begin{align}
		&\frac{e^{\lambda x^b}}{\ell(x) x^a } \sum_{n=1}^\infty \beta^{n}\widetilde{\mathbb{P}}\left(\max_{1\leq j \leq n} S_j \leq x, x-S_n\in [(k-1)\delta, k\delta) \right) \lesssim_\beta  \frac{e^{\lambda (k\delta)^b}}{\ell(k\delta) (k\delta)^a }.
	\end{align}
	(ii) For each $0\leq p< q$, 
	\begin{align}
		& \lim_{x\to+\infty} \frac{e^{\lambda x^b}}{\ell(x) x^a } \sum_{n=1}^\infty \beta^{n}\widetilde{\mathbb{P}}\left(\max_{1\leq j \leq n} S_j \leq x, x-S_n\in [p, q) \right) \nonumber\\
		&=\frac{1}{\mathbb{E}(e^{\gamma X})}\int_p^q \sum_{j=0}^\infty \frac{\beta^{j+1}}{1-\beta} \widetilde{\mathbb{P}}\left(z> -\min_{0\leq k\leq j} S_k \right)\mathrm{d} z. 
	\end{align}
\end{prop}

Since the proof of Proposition \ref{prop3} for $b\in (0,1)$ is complicated, so we only give the proof for $b=0$ here and the proof for $b\in (0,1)$ is postponed to Section \ref{Appendix}.  

\noindent
\textbf{Proof of Proposition \ref{prop3} for $b=0$:} We divide the proof into three steps.

\noindent
{\bf(Step 1)}. 
Let $I=[(k-1)\delta, k\delta)$ in the proof of (i) and $I= [p,q)$ in the proof of (ii). Recall the definition of $F_n$ in \eqref{def-of-F} for some $\varepsilon\in (1/2,1)$ close to $1$ such that $(-a)>\frac{2\varepsilon}{2\varepsilon -1}$.
We have 
\begin{align}\label{decom-J}
	& \sum_{n=1}^\infty  \beta^{n}\widetilde{\mathbb{P}}\left(\max_{1\leq j \leq n} S_j \leq x, x-S_n\in I\right)\nonumber\\
	& =  \sum_{n\geq (\log x)^2}  \beta^{n}\widetilde{\mathbb{P}}\left(\max_{1\leq j \leq n} S_j \leq x, x-S_n\in I\right) +  \sum_{n< (\log x)^2}  \beta^{n}\widetilde{\mathbb{P}}\left(\max_{1\leq j \leq n} S_j \leq x, x-S_n\in I, F_n^c \right)\nonumber\\
	&\qquad + \sum_{n < (\log x)^2}  \beta^{n}\widetilde{\mathbb{P}}\left(\max_{1\leq j \leq n} S_j \leq x, x-S_n\in I , F_n,  \min_{1\leq i\leq n} X_i<-x^{\varepsilon}\right) \nonumber\\
	& \qquad+  \sum_{n<(\log x)^2} \beta^{n}\widetilde{\mathbb{P}}\left(\max_{1\leq j \leq n} S_j \leq x, x-S_n\in I, F_n, \min_{1\leq i\leq n} X_i\geq -x^{\varepsilon}\right) =: \sum_{i=1}^4 J_i.
\end{align}

\noindent
{\bf Estimate for $J_1$.} It holds that 
\begin{align}\label{e24-1}
	& \frac{J_1}{\ell(x) x^a }   \leq \frac{1}{\ell(x) x^a } \sum_{n\geq (\log x)^2}  \beta^{n} \lesssim_\beta \frac{\beta^{(\log x)^2}}{\ell(x) x^a } . 
\end{align}

\noindent
{\bf Estimate for $J_2$.} By \eqref{e15-2}, we have 
\begin{align}\label{e24-2}
	& \frac{J_2}{\ell(x) x^a }  \leq \frac{1}{\ell(x) x^a } \sum_{n< (\log x)^2}  \beta^{n}\widetilde{\mathbb{P}}\left(F_n^c\right) \lesssim \sum_{n< (\log x)^2}  \beta^{n/2} \frac{\ell^2(x^{\varepsilon}) }{\ell(x) } x^{(2\varepsilon-1)a + 2\varepsilon}\nonumber\\
	&\lesssim_\beta \frac{\ell^2(x^{\varepsilon}) }{\ell(x) } x^{(2\varepsilon-1)a + 2\varepsilon}.
\end{align}

\noindent
{\bf Estimate for $J_3$.}  It follows from  \eqref{e15-4} that 
\begin{align}\label{e24-3}
	& \frac{J_3}{\ell(x) x^a } \lesssim \frac{e^{-\gamma x^{\varepsilon}}}{\ell(x) x^a } \sum_{n < (\log x)^2}  n\beta^{n} \lesssim_\beta \frac{e^{-\gamma x^{\varepsilon}}}{\ell(x) x^a } . 
\end{align}
Combining \eqref{e24-1}, \eqref{e24-2} and \eqref{e24-3}, we deduce that uniformly for all  interval $I\subset \mathbb{R}$, 
\begin{align}\label{e24-4}
	& \frac{J_1+J_2+J_3}{\ell(x) x^a }\lesssim_\beta 1\lesssim \frac{1}{\ell(k\delta)(k\delta)^a}\quad\mbox{and}\quad  \lim_{x\to+\infty} \frac{J_1+J_2+J_3}{\ell(x) x^a }  =0. 
\end{align}

\noindent
{\bf(Step 2)}. In this step, we prove (i).  

\noindent
{\bf Case 1: $k\delta \leq x/2$.}
 On the event $x-S_n <k\delta \leq x/2\ \Rightarrow\ S_n>x/2$, there exists a $1\leq j\leq n$ such that $X_j >x/(2n)> x/(2(\log x)^2)> x^\varepsilon$ for large $x$. Therefore, 
\begin{align}
	& J_4 = \sum_{n<(\log x)^2} \beta^{n}\widetilde{\mathbb{P}}\left(\max_{1\leq k \leq n} S_k \leq x, x-S_n\in I, \exists \  1\leq j\leq n, X_j>x^\varepsilon,  \max_{i\neq  j} |X_i|\leq x^{\varepsilon}\right) \nonumber\\
	& \leq \sum_{n<(\log x)^2} n\beta^{n}\widetilde{\mathbb{P}}\left( x-S_n\in [(k-1)\delta, k\delta), X_n>x^\varepsilon,  \max_{1\leq i\leq n-1} |X_i|\leq x^{\varepsilon}\right) \nonumber\\
	& \leq \sum_{n<(\log x)^2} n\beta^{n}\widetilde{\mathbb{P}}\left( x-S_n\in [(k-1)\delta, k\delta),    |S_{n-1}|\leq nx^{\varepsilon}\right) .
\end{align}
Noticing that $\{|S_{n-1}|\leq nx^{\varepsilon}\}\subset \{|S_{n-1}|\leq (\log x)^2 x^\varepsilon\}$ for all $n< (\log x)^2$, we conclude that 
\begin{align}
	& \frac{J_4}{\ell(x) x^a }   \leq \frac{1}{\ell(x)x^a}\sum_{n<(\log x)^2} n\beta^{n} \sup_{|z|\leq (\log x)^2 x^\varepsilon} \widetilde{\mathbb{P}}\left( X_n \in (x-z- k\delta, x-z-(k-1)\delta]  \right) \nonumber\\
	&\lesssim_\beta \frac{1}{\ell(x)x^a}  \sup_{|z|\leq (\log x)^2 x^\varepsilon} \int_{x-z-k\delta}^{x-z-(k-1)\delta} \ell(y)y^a \mathrm{d} y.
\end{align}
Since $x/4< x/2 -  (\log x)^2 x^\varepsilon  \leq x-z-k\delta \leq x+ (\log x)^2 x^\varepsilon<2x $ for large $x$ and that $\ell(y)y^a\lesssim \ell(x)x^a$ for all $y\in [x/4, 2x]$,  $J_4$ is bounded from above by
\begin{align}\label{e25-1}
	\frac{J_4}{\ell(x) x^a } \lesssim_\beta \frac{1}{\ell(x)x^a}  \sup_{|z|\leq (\log x)^2 x^\varepsilon} \int_{x-z-k\delta}^{x-z-(k-1)\delta} \ell(x)x^a \mathrm{d} y=\delta\leq  1. 
\end{align}

\noindent
{\bf Case 2: $x/2< k\delta <2 x $.}  In this case,  we naturally have $J_4\leq \sum_{n<(\log x)^2} \beta^{n}\lesssim_\beta 1$ and that 
\begin{align}\label{e25-2}
	\frac{J_4}{\ell(x) x^a } \lesssim_\beta \frac{1}{\ell(x) x^a }\lesssim \frac{1}{\ell(k\delta) (k\delta)^a}.
\end{align}

\noindent
{\bf Case 3: $k\delta \geq 2x$.}  In this case, we have for $x$ large enough,
\begin{align}\label{e25-3}
	&J_4 = \sum_{n<(\log x)^2} \beta^{n}\widetilde{\mathbb{P}}\left(\max_{1\leq k \leq n} S_j \leq x, x-S_n\in I, F_n, \min_{1\leq i\leq n} X_i\geq -x^{\varepsilon}\right) \nonumber\\
	& \leq \sum_{n<(\log x)^2} \beta^{n}\widetilde{\mathbb{P}}\left( S_n\leq x-(k-1)\delta, \min_{1\leq  i\leq n} X_i\geq -x^{\varepsilon}\right) \nonumber\\
	& \leq \sum_{n<(\log x)^2} \beta^{n}\widetilde{\mathbb{P}}\left( -nx^{\varepsilon}\leq S_n\leq \delta -x\right) \nonumber\\
	& \leq  \sum_{n<(\log x)^2} \beta^{n} \widetilde{\mathbb{P}}\left( -(\log x)^2x^{\varepsilon}\leq S_n\leq 1 -x\right) =0.
\end{align}
Now (i) follows immediately from \eqref{decom-J}, \eqref{e24-4}, \eqref{e25-1}, \eqref{e25-2} and \eqref{e25-3}.

\noindent
{\bf(Step 3)}. In this step, we prove (ii).  For each fixed $k,\delta$, let $x$ be large enough such that $k\delta \leq x/2$. Following the argument at the begining of {\bf Case 1}, we have 
\begin{align}\label{e36-1}
	& J_4 = \sum_{n<(\log x)^2} \beta^{n}\widetilde{\mathbb{P}}\left(\max_{1\leq k \leq n} S_k \leq x, x-S_n\in [p,q), \exists  \ 1\leq j\leq n, X_j>x^\varepsilon,  \max_{i\neq  j} |X_i|\leq x^{\varepsilon}\right) \nonumber\\
	& = \sum_{n<(\log x)^2} \beta^{n}\sum_{j=1}^n \widetilde{\mathbb{P}}\left(\max_{1\leq k \leq n} S_k \leq x, x-S_n\in [p,q),  X_j>x^\varepsilon,  \max_{i\neq  j} |X_i|\leq x^{\varepsilon}\right) .
\end{align}
Noticing that $\max_{1\leq k\leq j-1}S_k \leq nx^{\varepsilon}\leq x$ on the event $\{\max_{i\neq  j} |X_i|\leq x^{\varepsilon}\}$ for large $x$ and $n<(\log x)^2$. Therefore,  set $Y:= x-(S_n- X_j)$ and $W:= S_n-\max_{j\leq k\leq n} S_k$. Then $(Y,W)$ is independent of $X_j$ and it holds that  
\begin{align}
	&J_4 = \sum_{n<(\log x)^2} \beta^{n}\sum_{j=1}^n \widetilde{\mathbb{P}}\left(X_j\leq W+Y, X_j\in (Y-q,Y-p],  X_j>x^\varepsilon,  \max_{i\neq  j} |X_i|\leq x^{\varepsilon}\right) \nonumber\\
	& = \sum_{n<(\log x)^2} \beta^{n}\sum_{j=1}^n \widetilde{\mathbb{E}}\left(1_{\{\max_{i\neq  j} |X_i|\leq x^{\varepsilon} \}} \widetilde{\mathbb{P}}\left(X\leq w+y, X\in (y-q,y-p],  X>x^\varepsilon \right)\big|_{w=W, y=Y} \right).
\end{align}
Since on $\{\max_{i\neq  j} |X_i|\leq x^{\varepsilon} \}$, we have $Y-q >x- nx^{\varepsilon}-1 >x^{\varepsilon}$ for $x$ large enough (depending on $q$), we may drop the event $\{X>x^{\varepsilon}\}$ and get that 
\begin{align}\label{e36-3}
	&\lim_{x\to+\infty} \frac{J_4}{\ell(x)x^a}\nonumber\\
	&= \lim_{x\to+\infty} \frac{1}{\ell (x)x^a} \sum_{n<(\log x)^2} \beta^{n}\sum_{j=1}^n \widetilde{\mathbb{E}}\left(1_{\{\max_{i\neq  j} |X_i|\leq x^{\varepsilon} \}} \int_{-q}^{\min\{-p, W\}}  \frac{\ell(z+y)}{\mathbb{E}(e^{\gamma X})}(z+Y)^a\mathrm{d}z\right).
\end{align}
Noticing that on $\{\max_{i\neq  j} |X_i|\leq x^{\varepsilon} \}$, $|Y-x|\leq nx^{\varepsilon}< (\log x)^2 x^{\varepsilon}$, which implies that $\ell(z+Y)/\ell(x)\to 1$ and $(z+Y)/x\to 1$ as $x\to+\infty$ uniformly on $\{\max_{i\neq  j} |X_i|\leq x^{\varepsilon} \}$ and $z\in [-q, -p]$. Therefore, 
\begin{align}
	&\lim_{x\to+\infty} \frac{J_4}{\ell(x)x^a}\nonumber\\
	&= \lim_{x\to+\infty} \sum_{n<(\log x)^2} \beta^{n}\sum_{j=1}^n \widetilde{\mathbb{E}}\left(1_{\{\max_{i\neq  j} |X_i|\leq x^{\varepsilon} \}} \int_{-q}^{\min\{-p, W\}} \frac{1}{\mathbb{E}(e^{\gamma X})} \mathrm{d}z\right)\nonumber\\
	& = \frac{1}{\mathbb{E}(e^{\gamma X})} \sum_{n=1}^\infty \beta^{n}\sum_{j=1}^n \widetilde{\mathbb{E}}\left( \int_{-q}^{\min\{-p, W\}}1   \mathrm{d}z\right) = \frac{1}{\mathbb{E}(e^{\gamma X})} \int_{-q}^{-p} \sum_{n=1}^\infty \beta^{n}\sum_{j=1}^n \widetilde{\mathbb{P}}\left( z< S_n-\max_{j\leq k\leq n} S_k \right)   \mathrm{d}z.
\end{align}
Since $S_n- \max_{j\leq k\leq n}S_k$ is equal in law to $\min_{0\leq k\leq n-j} S_k$, we conclude that
\begin{align}
	&\lim_{x\to+\infty} \frac{J_4}{\ell(x)x^a} = \frac{1}{\mathbb{E}(e^{\gamma X})} \int_{-q}^{-p} \sum_{n=1}^\infty \beta^{n}\sum_{j=0}^{n-1} \widetilde{\mathbb{P}}\left( z < \min_{0\leq k\leq j} S_k \right)   \mathrm{d}z \nonumber\\
	& =  \frac{1}{\mathbb{E}(e^{\gamma X})} \int_p^q \sum_{j=0}^\infty \frac{\beta^{j+1}}{1-\beta} \widetilde{\mathbb{P}}\left(z> -\min_{0\leq k\leq j} S_k \right)\mathrm{d} z. 
\end{align}
Combining \eqref{decom-J}, \eqref{e24-4} and the above limit, we get (ii). 

\hfill$\Box$

\section{Proof of the main results}\label{S4}

The following lemma can be found in \cite[p.23]{HZ2025}.

\begin{lemma}\label{lem3}
	Define $\psi(u):=\sum_{n=0}^\infty p_n(1-u)^n+ mu-1$ with the convension $0^0=1$. Then the function $u(x):= \mathbb{P}(M>x)$ satisfies
	\begin{align}
		u(x)=\frac{1-p_0}{m}\sum_{n=1}^\infty m^{n} \mathbb{P}(T_x^+=n)-\sum_{n=1}^\infty m^{n-1}\mathbb{E}\left( \psi\left(u(x-S_n)\right)1_{\{T_x^+>n\}}\right).
	\end{align}
\end{lemma}

\subsection{Proof of Theorem \ref{thm1}}\label{S4.1}

\textbf{Proof of Theorem \ref{thm1}:} Combining Proposition \ref{prop1}, Lemma \ref{lem3} and dominated convergence theorem, we see that 
\begin{align}\label{e12-4}
	\lim_{x\to+\infty} \sum_{n=1}^\infty m^{n} \frac{ \mathbb{P}(T_x^+=n)}{\mathbb{P}(X>x)}=\sum_{n=1}^\infty m^{n} =\frac{m}{1-m}.
\end{align}
For the non-linear term, noticing that $\lim_{u\to 0+}\frac{\psi(u)}{u}=0$, for any $\epsilon>0$, there exists $\delta>0$ such that $\psi(u)\leq \epsilon u$ for all $u\in (0,\delta)$. Therefore, we may fix a large $N$ such that $u(y)<\delta$ when $y>N$.  Since $\psi(u)\leq 1$ for all $u\in [0,1]$, we obtain that 
\begin{align}\label{e12-1}
	& \sum_{n=1}^\infty m^{n-1}\mathbb{E}\left(\psi\left(u(x-S_n)\right)1_{\{T_x^+>n\}}\right) \nonumber\\
	& \leq \epsilon \sum_{n=1}^\infty m^{n-1}\mathbb{E}\left(u(x-S_n)1_{\{T_x^+>n, S_n<x-N\}}\right) +\sum_{n=1}^\infty m^{n-1}\mathbb{P}\left(T_x^+>n, S_n\in (x-N, x]\right) \nonumber\\
	& \leq \epsilon \sum_{n=1}^\infty m^{n-1}\mathbb{E}\left(u(x-S_n)1_{\{T_x^+>n, S_n<x-N\}}\right) +\sum_{n=1}^\infty m^{n-1}\mathbb{P}\left(S_n\in (x-N, x]\right).
\end{align}
For the second term, combining Proposition \ref{prop1} (i), \eqref{e6} and dominated convergence theorem,  
\begin{align}\label{e12-2}
	& \lim_{x\to+\infty} \sum_{n=1}^\infty m^{n-1}\frac{\mathbb{P}\left(S_n\in (x-N, x]\right)}{\mathbb{P}(X>x)} \nonumber\\
	& = \lim_{x\to+\infty} \sum_{n=1}^\infty m^{n-1}\frac{\mathbb{P}\left(S_n >x-N\right)}{\mathbb{P}(X>x)}  -\lim_{x\to+\infty} \sum_{n=1}^\infty m^{n-1}\frac{\mathbb{P}\left(S_n> x\right)}{\mathbb{P}(X>x)} \nonumber\\
	& = \sum_{n=1}^\infty n m^{n-1}- \sum_{n=1}^\infty n m^{n-1}=0.
\end{align}
For the first term, combining Proposition \ref{prop1}, Lemma \ref{lem3} and the fact that $\psi\geq 0$, we have 
\[
u(x)\leq \frac{1-p_0}{m}\sum_{n=1}^\infty m^{n} \mathbb{P}(T_x^+=n) \lesssim \mathbb{P}(X>x).
\]
Therefore, we conclude that 
\begin{align}\label{e12-3}
	&\epsilon \sum_{n=1}^\infty m^{n-1}\mathbb{E}\left(u(x-S_n)1_{\{T_x^+>n, S_n<x-N\}}\right) \nonumber\\
	& \lesssim \epsilon \sum_{n=1}^\infty m^{n-1}\mathbb{E}\left(\mathbb{P}\left( X>y\right)\big|_{y=x-S_n}1_{\{T_x^+>n, S_n<x-N\}}\right) \nonumber\\
	& = \epsilon \sum_{n=1}^\infty m^{n-1}\mathbb{P}\left( S_{n+1}>x, T_x^+>n, S_n<x-N\right) \nonumber\\
	&\leq  \epsilon \sum_{n=1}^\infty m^{n-1}\mathbb{P}\left(  T_x^+=n+1 \right) \lesssim \epsilon \mathbb{P}(X>x).
\end{align}
Combining \eqref{e12-1}, \eqref{e12-2} and \eqref{e12-3}, we conclude that 
\begin{align}\label{e12-5}
	\limsup_{x\to+\infty} \frac{1}{\mathbb{P}(X>x)} \sum_{n=1}^\infty m^{n-1}\mathbb{E}\left(\psi\left(u(x-S_n)\right)1_{\{T_x^+>n\}}\right) \lesssim \epsilon \stackrel{\epsilon\to 0}{\longrightarrow}0.
\end{align}
Now it follows from Lemma \ref{lem3}, \eqref{e12-4} and \eqref{e12-5} that 
\begin{align}
	\lim_{x\to+\infty} \frac{u(x)}{\mathbb{P}(X>x)} = \frac{1-p_0}{m}\times \frac{m}{1-m}= \frac{1-p_0}{1-m},
\end{align}
which completes the proof of the theorem. 

\hfill$\Box$

\subsection{Proof of Theorem \ref{thm2}}\label{S4.2}

Set $\varphi(u):= \frac{\psi(u)}{u}$. Under the assumption $\sum_{k=1}^\infty k(\log k)p_k<\infty$, it follows from \cite[Lemma 7]{HRSZ25} that $\varphi(u)$ is increasing in $u\in [0,1]$ and  that for any $c>0$,
\begin{align}\label{proper-of-phi}
	\int_0^\infty \varphi (e^{-ct})\mathrm{d} t<\infty .
\end{align}

 \begin{lemma}\label{lem4}
 	Assume $\sum_{k=1}^\infty k(\log k) p_k<\infty$. Then $\int_1^\infty \varphi(u(x))\mathrm{d} x<\infty$.  Moreover, for each $\delta>0$, we have
 	\[
 	\sum_{k=1}^\infty \varphi(u(k\delta))<\infty. 
 	\]
 \end{lemma}
 \textbf{Proof:} According to the  monotonicity property of $\varphi$ and \eqref{proper-of-phi}, it remains to prove that there exists some $c>0$ such that $u(x)\leq e^{-cx}$ for all $x>K_0$.
 Set $\beta := m \mathbb{E}(e^{\gamma X})\in (0,1)$.  Combining \eqref{Change-of-measure} and Lemma \ref{lem3}, we have
 \begin{align}\label{eq-of-u}
 	u(x)=\frac{1-p_0}{m}\sum_{n=1}^\infty \beta^{n} \widetilde{\mathbb{E}}\left(e^{-\gamma S_n}1_{\{T_x^+=n\}}\right)-\frac{1}{m}\sum_{n=1}^\infty \beta^{n}\widetilde{\mathbb{E}}\left(e^{-\gamma S_n} \psi\left(u(x-S_n)\right)1_{\{T_x^+>n\}}\right).
 \end{align}
 Since $\{T_x^+=n\} \subset \{S_n>x\}$ and $m=\sum_{k=1}^\infty kp_k \geq \sum_{k=1}^\infty p_k =1-p_0$, we conclude from \eqref{eq-of-u} that 
  \begin{align} 
  	& u(x) \leq  \sum_{n=1}^\infty  \widetilde{\mathbb{E}}\left(e^{-\gamma S_n}1_{\{T_x^+=n\}}\right) \leq  e^{-\gamma x}\sum_{n=1}^\infty   \widetilde{\mathbb{P}}\left( T_x^+=n \right) \leq e^{-\gamma x},
  \end{align}
 which implies the desired result.

 \hfill$\Box$

\noindent
\textbf{Proof of Theorem \ref{thm2}:}  
Combining Proposition \ref{prop2} and the dominated convergence theorem, 
\begin{align}\label{e23-1}
	& \lim_{x\to+\infty} \frac{e^{\lambda x^b+ \gamma x}}{\ell(x)x^a  } \sum_{n=1}^\infty \beta^{n} \widetilde{\mathbb{E}}\left(e^{-\gamma S_n}1_{\{T_x^+=n\}}\right)\nonumber\\
	& = \frac{1}{\gamma \mathbb{E}(e^{\gamma X})} \left(1+ \sum_{j=1}^{\infty}\beta^j \widetilde{\mathbb{E}}\left( \left(1-e^{-\gamma \min_{1\leq i\leq j} S_i}\right)_+\right) \right)\frac{\beta}{1-\beta}.
\end{align}
For the non-linear term of \eqref{eq-of-u}, noticing that for each $\delta\in (0,1)$,
\begin{align}
	&\sum_{n=1}^\infty \beta^{n}\widetilde{\mathbb{E}}\left(e^{-\gamma S_n} \psi\left(u(x-S_n)\right)1_{\{T_x^+>n\}}\right)\nonumber\\
	&=e^{-\gamma x}\sum_{n=1}^\infty \beta^{n}\widetilde{\mathbb{E}}\left(e^{\gamma (x- S_n)} \psi\left(u(x-S_n)\right)1_{\{\max_{1\leq j \leq n} S_j \leq x\}}\right)\nonumber\\
	& = e^{-\gamma x}\sum_{k=1}^\infty \sum_{n=1}^\infty \beta^{n}\widetilde{\mathbb{E}}\left(e^{\gamma (x- S_n)} \psi\left(u(x-S_n)\right)1_{\{\max_{1\leq j \leq n} S_j \leq x, x-S_n\in [(k-1)\delta, k\delta) \}}\right).
\end{align}
Since $\psi(u(z))$ is decreasing in $z$, we obtain that 
\begin{align}\label{e31-1}
	& \sum_{k=1}^\infty e^{\gamma(k-1)\delta}\psi(u(k\delta))\sum_{n=1}^\infty \beta^{n}\widetilde{\mathbb{P}}\left( \max_{1\leq j \leq n} S_j \leq x, x-S_n\in [(k-1)\delta, k\delta) \right)\nonumber\\
	&  \leq e^{\gamma x} \sum_{n=1}^\infty \beta^{n}\widetilde{\mathbb{E}}\left(e^{-\gamma S_n} \psi\left(u(x-S_n)\right)1_{\{T_x^+>n\}}\right)  \nonumber\\
	&\leq  \sum_{k=1}^\infty e^{\gamma k\delta}\psi(u((k-1)\delta))\sum_{n=1}^\infty \beta^{n}\widetilde{\mathbb{P}}\left(\max_{1\leq j \leq n} S_j \leq x, x-S_n\in [(k-1)\delta, k\delta) \right).
\end{align}
It follows from \eqref{eq-of-u} and \eqref{e23-1} that $u(x)\lesssim e^{-\gamma x} \ell(x)x^a e^{-\lambda x^b}$. Therefore, 
\begin{align}\label{e31-2}
	e^{\gamma k\delta} \psi(u(k\delta))= e^{\gamma k\delta} u(k\delta) \varphi(u(k\delta)) \lesssim \ell(k\delta)(k\delta)^a e^{-\lambda (k\delta)^b}  \varphi(u(k\delta)).
\end{align}
Combining Proposition \ref{prop3}, \eqref{e31-2}, Lemma \ref{lem4} and dominated convergence theorem, for each $\delta\in (0,1)$,
\begin{align}\label{e31-3}
	& \frac{1}{\mathbb{E}(e^{\gamma X})} \sum_{k=1}^\infty e^{\gamma (k-1)\delta}\psi(u(k\delta)) \int_{(k-1)\delta}^{k\delta} \sum_{j=0}^\infty \frac{\beta^{j+1}}{1-\beta} \widetilde{\mathbb{P}}\left(z> -\min_{0\leq k\leq j} S_k \right)\mathrm{d} z\nonumber\\
	&=: \sum_{k=1}^\infty \Xi_1(k, \delta)  \leq \liminf_{x\to+\infty} \frac{e^{\lambda x^b+\gamma x}}{\ell(x)x^a } \sum_{n=1}^\infty \beta^{n}\widetilde{\mathbb{E}}\left(e^{-\gamma S_n} \psi\left(u(x-S_n)\right)1_{\{T_x^+>n\}}\right) \nonumber\\
	&\leq \limsup_{x\to+\infty} \frac{e^{\lambda x^b+\gamma x}}{\ell(x)x^a } \sum_{n=1}^\infty \beta^{n}\widetilde{\mathbb{E}}\left(e^{-\gamma S_n} \psi\left(u(x-S_n)\right)1_{\{T_x^+>n\}}\right)  \leq \sum_{k=1}^\infty \Xi_2(k, \delta)  \nonumber\\
	&:= \frac{1}{\mathbb{E}(e^{\gamma X})} \sum_{k=1}^\infty e^{\gamma k\delta}\psi(u((k-1)\delta)) \int_{(k-1)\delta}^{k\delta} \sum_{j=0}^\infty \frac{\beta^{j+1}}{1-\beta} \widetilde{\mathbb{P}}\left(z> -\min_{0\leq k\leq j} S_k \right)\mathrm{d} z.
\end{align}
Now we would like to take $\delta\to0$ in  \eqref{e31-3}. We only treat   $\Xi_1(k,\delta)$ here since the proof for $\Xi_2(k,\delta)$  is similar. Fix $N>0$, let $L\in \mathbb{N}$ be large enough such that  $\delta:= N/L \in (0,1)$. Combining Lemma \ref{lem4}, \eqref{e31-2} and the fact that $\psi(u(x))$ is decreasing in $x$, we have 
\begin{align}
	&\sum_{k=L+1}^\infty \Xi_1(k, \delta) \leq \frac{1}{\mathbb{E}(e^{\gamma X})} \sum_{j=0}^\infty \frac{\beta^{j+1}}{1-\beta}  \sum_{k=L+1}^\infty  \int_{(k-1)\delta}^{k\delta} e^{\gamma x}\psi(u(x)) \mathrm{d} z\nonumber\\
	&=  \frac{1}{\mathbb{E}(e^{\gamma X})}\sum_{j=0}^\infty \frac{\beta^{j+1}}{1-\beta}  \int_{N}^{\infty} e^{\gamma x}\psi(u(x))\mathrm{d} z  \lesssim  \int_N^\infty \varphi(u(x))\mathrm{d}x.
\end{align}
Therefore, for any $\varepsilon>0$, we may choose a large $N$ such that 
\begin{align}\label{e32-1}
	& \sum_{k=L+1}^\infty \Xi_1(k, \delta)\leq  \frac{1}{\mathbb{E}(e^{\gamma X})} \sum_{j=0}^\infty \frac{\beta^{j+1}}{1-\beta}  \int_{N}^{\infty} e^{\gamma x}\psi(u(x))\mathrm{d} z  <\varepsilon.
\end{align}
For the case $1\leq k\leq L$,  since $\psi(u(x))$ is decreasing in $x$, we have the following upper bound
\begin{align}\label{e32-2}
	& \sum_{k=1}^L \Xi_1(k, \delta)\leq \frac{1}{\mathbb{E}(e^{\gamma X})} \sum_{k=1}^L  \int_{(k-1)\delta}^{k\delta} e^{\gamma x}\psi(u(x))  \sum_{j=0}^\infty \frac{\beta^{j+1}}{1-\beta} \widetilde{\mathbb{P}}\left(z>-\min_{0\leq k\leq j} S_k \right)\mathrm{d} z \nonumber\\
	& =\frac{1}{\mathbb{E}(e^{\gamma X})} \int_0^N e^{\gamma x}\psi(u(x))  \sum_{j=0}^\infty \frac{\beta^{j+1}}{1-\beta} \widetilde{\mathbb{P}}\left(z> -\min_{0\leq k\leq j} S_k \right)\mathrm{d} z. 
\end{align}
For the lower bound, we also have 
\begin{align}\label{e32-3}
	& \sum_{k=1}^L  \Xi_1(k, \delta) \geq\frac{e^{-\gamma\delta}}{\mathbb{E}(e^{\gamma X})} \sum_{k=1}^L   \int_{(k-1)\delta}^{k\delta}e^{\gamma x}\psi(u(x-\delta)) \sum_{j=0}^\infty \frac{\beta^{j+1}}{1-\beta} \widetilde{\mathbb{P}}\left(z> -\min_{0\leq k\leq j} S_k \right)\mathrm{d} z\nonumber\\
	& =  \frac{e^{-\gamma \delta}}{\mathbb{E}(e^{\gamma X})} \int_0^N e^{\gamma x}\psi(u(x-\delta)) \sum_{j=0}^\infty \frac{\beta^{j+1}}{1-\beta} \widetilde{\mathbb{P}}\left(z> -\min_{0\leq k\leq j} S_k \right)\mathrm{d} z.
\end{align}
Noticing that $u(x)$ is decreasing in $x$, we have $\psi(u(x-\delta))\to \psi(u(x-))$ as $\delta\to 0$. Since $\psi(u(x-))=\psi(u(x))$ for a.e. $x\in [0,N]$, 
taking $L\to\infty $ in both \eqref{e32-2} and \eqref{e32-3}, it follows from dominated convergence theorem (in \eqref{e32-3}) that 
\begin{align}
	& \lim_{L\to\infty}  \sum_{k=1}^L  \Xi_1(k, \delta)  = \frac{1}{\mathbb{E}(e^{\gamma X})} \int_0^N e^{\gamma x}\psi(u(x))  \sum_{j=0}^\infty \frac{\beta^{j+1}}{1-\beta} \widetilde{\mathbb{P}}\left(z> -\min_{0\leq k\leq j} S_k \right)\mathrm{d} z. 
\end{align}
Combining \eqref{e32-1} and the above limit, we obtain that 
\begin{align}\label{e32-4}
	& \lim_{L\to\infty}\sum_{k=1}^\infty \Xi_1(k, \delta) =\frac{1}{\mathbb{E}(e^{\gamma X})} \int_0^\infty e^{\gamma x}\psi(u(x))  \sum_{j=0}^\infty \frac{\beta^{j+1}}{1-\beta} \widetilde{\mathbb{P}}\left(z> -\min_{0\leq k\leq j} S_k \right)\mathrm{d} z. 
\end{align}
Similarly, we also have that 
\begin{align}\label{e32-5}
	& \lim_{L\to\infty} \sum_{k=1}^\infty \Xi_2(k, \delta)=\frac{1}{\mathbb{E}(e^{\gamma X})} \int_0^\infty e^{\gamma x}\psi(u(x))  \sum_{j=0}^\infty \frac{\beta^{j+1}}{1-\beta} \widetilde{\mathbb{P}}\left(z> -\min_{0\leq k\leq j} S_k \right)\mathrm{d} z. 
\end{align}
Plugging \eqref{e32-4} and \eqref{e32-5} back to \eqref{e31-1} yields that 
\begin{align}\label{e33-1}
	& \lim_{x\to+\infty} \frac{e^{\lambda x^b+ \gamma x} }{\ell(x)x^a }\sum_{n=1}^\infty \beta^{n}\widetilde{\mathbb{E}}\left(e^{-\gamma S_n} \psi\left(u(x-S_n)\right)1_{\{T_x^+>n\}}\right)\nonumber\\
	& = \frac{1}{\mathbb{E}(e^{\gamma X})} \int_0^\infty e^{\gamma x}\psi(u(x))  \sum_{j=0}^\infty \frac{\beta^{j+1}}{1-\beta} \widetilde{\mathbb{P}}\left(z> -\min_{0\leq k\leq j} S_k \right)\mathrm{d} z.
\end{align}
Combining \eqref{eq-of-u}, \eqref{e23-1} and \eqref{e33-1}, we conclude that 
\begin{align}\label{expression-of-C}
	& \lim_{x\to+\infty} \frac{e^{\lambda x^b+ \gamma x}}{\ell(x)x^a  } u(x)\nonumber\\
	& = \frac{1-p_0}{m}\frac{1}{\gamma \mathbb{E}(e^{\gamma X})} \left(1+ \sum_{j=1}^{\infty}\beta^j \widetilde{\mathbb{E}}\left( \left(1-e^{-\gamma \min_{1\leq i\leq j} S_i}\right)_+\right) \right)\frac{\beta}{1-\beta}\nonumber\\
	&\qquad - \frac{1}{m\mathbb{E}(e^{\gamma X})} \int_0^\infty e^{\gamma x}\psi(u(x))  \sum_{j=0}^\infty \frac{\beta^{j+1}}{1-\beta} \widetilde{\mathbb{P}}\left(z> -\min_{0\leq k\leq j} S_k \right)\mathrm{d} z=: C_0\in [0, \infty). 
\end{align}
It remains to show that $C_0>0$. Since $u(x)\geq (1-p_0)\mathbb{P}(X>x) \geq (1-p_0)\mathbb{P}(X\in (x,x+1))$, we obtain
\begin{align}
	\frac{e^{\lambda x^b+ \gamma x}}{\ell(x)x^a } u(x) \geq (1-p_0) \frac{e^{\lambda x^b+\gamma x}}{\ell(x)x^a  }  \int_x^{x+1}\ell(y)y^ae^{-\lambda y^b-\gamma y}\mathrm{d} y\gtrsim 1,
\end{align}
as desired. 

\hfill$\Box$

\section{Proof of Proposition \ref{prop3} for $b\in (0,1)$}\label{Appendix}

We divide the proof of Proposition \ref{prop3} for $b\in (0,1)$ into several lemmas. 
Recall the definition of $H_n$ in \eqref{def-of-H} for $\varepsilon_0= \frac{1}{2}\min\{ b, 1-b\}$. 
We have the following decomposition for any interval $I$:
\begin{align}\label{decomposition}
	&  \sum_{n=1}^\infty \beta^{n}\widetilde{\mathbb{P}}\left(\max_{1\leq j \leq n} S_j \leq x, x-S_n\in I \right) \nonumber\\
	&=  \sum_{n\geq x^{(1+b)/2}} \beta^{n}\widetilde{\mathbb{P}}\left(\max_{1\leq j \leq n} S_j \leq x, x-S_n\in I \right) + \sum_{n<x^{(1+b)/2}}  \beta^{n}\widetilde{\mathbb{P}}\left(\max_{1\leq j \leq n} S_j \leq x, x-S_n\in I, H_n^c\right) \nonumber\\
	&\qquad + \sum_{(\log x)^{5/2}\leq n <x^{(1+b)/2}} \beta^{n}\widetilde{\mathbb{P}}\left(\max_{1\leq j \leq n} S_j \leq x, x-S_n\in I, H_n \right) \nonumber\\
	&\qquad + \sum_{n< (\log x)^{5/2} } \beta^{n}\widetilde{\mathbb{P}}\left(\max_{1\leq j \leq n} S_j \leq x, x-S_n\in I, H_n \right) \nonumber\\
	&=: \sum_{i=1}^4 T_i. 
\end{align}

\begin{lemma}\label{lem5-1}
	Assume {\bf(H2)} and $b\in (0,1)$. 
	\begin{itemize}
		\item[(i)] We have 
		\begin{align}
			\sup_{I\subset [0,\infty)}  \frac{e^{\lambda x^b}}{\ell(x)x^a }(T_1+T_2) \lesssim_\beta  e^{-\lambda x^{b}/2}.
		\end{align}
		\item[(ii)] Fix $\delta\in (0,1)$. For any $k\delta >x$ and $I= [(k-1)\delta, k\delta)$, it holds that 
		\begin{align}
		\frac{e^{\lambda x^b}}{\ell(x)x^a }(T_3+T_4)  \lesssim_\beta \frac{e^{ \lambda (k\delta)^b }}{\ell(k\delta)(k\delta)^a}.
		\end{align}
	\end{itemize}
\end{lemma}
\textbf{Proof: }
(i) By \eqref{e19-2}, it holds that 
\begin{align}
	& \frac{e^{\lambda x^b}}{\ell(x)x^a }(T_1+T_2)  \leq  \frac{e^{\lambda x^b}}{\ell(x)x^a } \sum_{n\geq x^{(1+b)/2}} \beta^{n} +\sum_{n=1}^\infty \beta^{n/2} \frac{ \beta^{n/2}e^{\lambda x^b}\widetilde{\mathbb{P}}\left(  H_n^c\right)}{\ell(x)x^a } \nonumber\\
	&\lesssim_\beta \frac{\beta^{x^{(1+b)/2}}e^{\lambda x^b}}{\ell(x)x^a }+ e^{-\lambda x^b/2} \nonumber\\
	&\lesssim_\beta  e^{-\lambda x^b/2},
\end{align}
as desired. 

(ii) In the case $k\delta>x$, we have 
\begin{align}
	\frac{e^{\lambda x^b}}{\ell(x)x^a } (T_3+T_4)\leq \frac{e^{\lambda x^b}}{\ell(x)x^a } \sum_{n=1}^\infty   \beta^{n} \lesssim_\beta   \frac{e^{\lambda x^b}}{\ell(x)x^a }.
\end{align}
Since  $\frac{e^{\lambda x^b}}{\ell(x)x^a } \lesssim \frac{e^{\lambda x^b}}{\ell(k\delta)(k\delta)^a } \leq \frac{e^{\lambda (k\delta)^b}}{\ell(k\delta)(k\delta)^a } $ when $k\delta \in (x,2x)$ and  $\frac{e^{\lambda x^b}}{\ell(x)x^a }  \lesssim  e^{\lambda (3x/2)^b} \leq e^{\lambda (3 k\delta/4)^b} \lesssim  \frac{e^{\lambda (k\delta)^b}}{\ell(k\delta)(k\delta)^a } $ when $k\delta  \geq 2x$, we conclude that for $k\delta>x$,
\begin{align}\label{e26-3}
	& \frac{e^{\lambda x^b}}{\ell(x)x^a } (T_3+T_4) \lesssim_\beta  \frac{e^{ \lambda (k\delta)^b }}{\ell(k\delta)(k\delta)^a},
\end{align}
which implies (ii). 

\hfill$\Box$

\begin{lemma}\label{lem5-2}
	Assume {\bf(H2)} and $b\in (0,1)$.  
	\begin{itemize}
		\item[(i)] Fix $\delta\in (0,1)$. Then for all $k\delta \geq K_0, I= [(k-1)\delta, k\delta)$ and $x>K_0$, 
		\begin{align}
			\frac{e^{\lambda x^b}}{\ell(x)x^a }T_3 \lesssim_{\beta } \frac{e^{\lambda (k\delta)^b}}{\ell(k\delta)(k\delta)^a} .
		\end{align}
		\item[(ii)] For $I=[p, q)$, it holds that 
		\begin{align}
			\lim_{x\to+\infty} \frac{e^{\lambda x^b}}{\ell(x)x^a }T_3  =0.
		\end{align}
	\end{itemize}
\end{lemma}
\textbf{Proof: }  
By Lemma \ref{lem5-1} (ii), it suffices to prove (i) for $K_0\leq k\delta \leq x$ and $x>K_0$. Since 
\begin{align}
	T_3\leq \sum_{(\log x)^{5/2}\leq n <x^{(1+b)/2}} \beta^{n}\widetilde{\mathbb{P}}\left( x-S_n\in I, H_n \right) =: \widehat{T}_3. 
\end{align}
Thus it remains to prove (i) and (ii) with $T_3$ replaced by $\widehat{T}_3$. 
According to the definition of $H_n$ in \eqref{def-of-H}, we have the following upper bound 
\begin{align}\label{decomposition-T-3}
	&\widehat{T}_3 \leq \sum_{(\log x)^{5/2}\leq n <x^{(1+b)/2}} \beta^{n}\widetilde{\mathbb{P}}\left( x-S_n\in I, \max_{1\leq i\leq n} X_i \leq x-2x^{\varepsilon_0}  \right)  \nonumber\\
	&\qquad +\sum_{(\log x)^{5/2}\leq n <x^{(1+b)/2}} \beta^{n}\widetilde{\mathbb{P}}\left( x-S_n\in I, \exists ! \ 1\leq j\leq n, X_j>x-2x^{\varepsilon_0} , |S_n-X_j| >x^{\varepsilon_0}\right) \nonumber\\
	&\qquad +\sum_{(\log x)^{5/2}\leq n <x^{(1+b)/2}} \beta^{n}\widetilde{\mathbb{P}}\left( x-S_n\in I, \exists ! \ 1\leq j\leq n, X_j>x-2x^{\varepsilon_0} , |S_n-X_j| \leq x^{\varepsilon_0} \right) \nonumber\\
	&=: \widehat{T}_{31}+ \widehat{T}_{32}+ \widehat{T}_{33}.
\end{align}

\noindent
{\bf Estimate for $\widehat{T}_{31}$}. Let $z= k\delta \in [1,x]$ in the proof of (i) and $z=q$ in the proof of (ii). We may assume also in (ii) that $x$ is large enough such that $z\leq x$. 
Combining the same argument as \eqref{e33-2} (with the event $S_n>x$ replaced by $S_n>x-z$) and Lemma  \ref{lem2}, we have 
\begin{align}
	& \frac{e^{\lambda x^b}}{\ell(x)x^a } \widehat{T}_{31} \leq \frac{e^{\lambda x^b}}{\ell(x)x^a } \sum_{(\log x)^{5/2}\leq n<x^{(1+b)/2}}  \beta^{n}\widetilde{\mathbb{P}}\left( S_n> x-z, \max_{1\leq i\leq n} X_i \leq x-2x^{\varepsilon_0}\right) \nonumber\\
	& \leq \frac{e^{\lambda x^b}}{\ell(x)x^a } \sum_{(\log x)^{5/2}\leq n<x^{(1+b)/2}}  \beta^{n} e^{-\theta(x)(x-z -n\widetilde{\mu})+C_*\theta^2(x)n} \nonumber\\
	& = \frac{e^{\theta(x)z}}{\ell(x)x^a } \sum_{(\log x)^{5/2}\leq n<x^{(1+b)/2}}  \beta^{n} e^{ (\log x)^2+ (\widetilde{\mu} \theta(x)+C_*\theta^2(x))n} ,
\end{align}
where $\widetilde{\mu}:= \widetilde{\mathbb{E}}(X)$. Applying inequality $\ell (x) x^{a}\gtrsim x^{-|a|-1}$ and  taking $x$
sufficiently large such that 
$n^{-1/5}+ (|a|+1)n^{-3/5}+ C_*\theta^2(x)+\widetilde{\mu} \theta(x)\leq -\frac{1}{3}\log \beta$ for all $n\geq  (\log x)^{5/2}$,  we obtain 
\begin{align}\label{e26-1}
	& \frac{e^{\lambda x^b}}{\ell(x)x^a } \widehat{T}_{31} \lesssim e^{ \lambda x^{b-1}z}  \sum_{(\log x)^{5/2}\leq n<x^{(1+b)/2}}  \beta^{n} e^{ (|a|+1)\log x+ (\log x)^2+ (\widetilde{\mu} \theta(x)+C_*\theta^2(x))n} \nonumber\\
	& \leq e^{ \lambda x^{b-1}z}  \sum_{(\log x)^{5/2}\leq n<x^{(1+b)/2}}  \beta^{n} e^{  (n^{-1/5}+(|a|+1)n^{-3/5}+\widetilde{\mu} \theta(x)+C_*\theta^2(x))n} \nonumber\\
	& \leq e^{ \lambda x^{b-1}z}  \sum_{(\log x)^{5/2}\leq n<x^{(1+b)/2}}  \beta^{2n/3} \nonumber\\
	&\lesssim_\beta  e^{ \lambda  z^b } \beta ^{2(\log x)^{5/2}/3},
\end{align}
where in the last inequality we used the inequality $z\leq x$.
Therefore, in the proof of (ii), we have 
\begin{align}\label{e26-2'}
	\limsup_{x\to+\infty}   \frac{e^{\lambda x^b}}{\ell(x)x^a } \widehat{T}_{31} \lesssim_\beta  \lim_{x\to+\infty}  e^{ \lambda  q^b } \beta ^{2(\log x)^{5/2}/3}=0.
\end{align}
For (i), using the fact  that $\frac{1}{\ell(z) z^a}\gtrsim z^{-|a|-1}\geq x^{-|a|-1}$ for $1\leq z \leq x$, we conclude that for $k\delta \leq x $,
\begin{align}\label{e26-2}
	& \frac{e^{\lambda x^b}}{\ell(x)x^a } \widehat{T}_{31} \lesssim_\beta \frac{e^{ \lambda (k\delta)^{b}  }}{\ell(k\delta)(k\delta)^a} x^{|a|+1} \beta ^{2(\log x)^{5/2}/3} \lesssim_\beta  \frac{e^{ \lambda (k\delta)^{b}  }}{\ell(k\delta)(k\delta)^a} .
\end{align}

\noindent
{\bf Estimate for $\widehat{T}_{32}$}.  Noticing that 
\begin{align}
	&  \frac{e^{\lambda x^b}}{\ell(x)x^a } \widehat{T}_{32}   \leq \frac{e^{\lambda x^b}}{\ell(x)x^a } \sum_{(\log x)^{5/2}\leq n<x^{(1+b)/2}}  n \beta^{n}\widetilde{\mathbb{P}}\left(X_n > x-2x^{\varepsilon_0}, |S_n-X_n|>x^{\varepsilon_0}\right) \nonumber\\
	& \leq \frac{e^{\lambda x^b}}{\ell(x)x^a } \widetilde{\mathbb{P}}(X>x-2x^{\varepsilon_0}) \sum_{(\log x)^{5/2}\leq n<x^{(1+b)/2}}  n^2 \beta^{n} \widetilde{\mathbb{P}}(|X|>x^{\varepsilon_0}/n).
\end{align}
Since $\lim_{x\to+\infty} ((x-2x^{\varepsilon_0})^b-x^b)=0$, combining Lemma \ref{lem1}, Markov's inequality and the fact that $\widetilde{\mathbb{E}}(|X|^k)<\infty$ for all $k>0$,  let $k=k_0=1/\varepsilon_0$, then $k_0 \varepsilon_0 >1-b$ and the above inequality is bounded from above by
\begin{align}\label{e30-1}
	& \frac{e^{\lambda x^b}}{\ell(x)x^a } \widehat{T}_{32} \lesssim  \frac{\ell(x) x^{a+1-b}}{\ell(x)x^a } \frac{\widetilde{\mathbb{E}}(|X|^{k_0})}{(x^{\varepsilon_0})^{k_0}} \sum_{n=1}^\infty n^{k_0+2} \beta^n \lesssim_\beta  x^{-b}.
\end{align}

\noindent
{\bf Estimate for $\widehat{T}_{33}$}.
If $I=[(k-1)\delta, k\delta)$, then we have the following upper bound for $\widehat{T}_{33}$:
\begin{align}\label{e27-2}
	&  \widehat{T}_{33} \leq   \sum_{(\log x)^{5/2}\leq n<x^{(1+b)/2}}  n \beta^{n} \widetilde{\mathbb{P}}\left(x-3x^{\varepsilon_0}<S_n\leq x-k\delta+1 \right).
\end{align}
Therefore, if $k\delta-1 >3x^{\varepsilon_0}$, then $\widehat{T}_{33} =0$.  If $k\delta\leq 3x^{\varepsilon_0}+1$ then
\begin{align}
	&  \widehat{T}_{33} \leq     \sum_{(\log x)^{5/2}\leq n<x^{(1+b)/2}}  n \beta^{n} \widetilde{\mathbb{P}}\left(S_n\in (x-k\delta, x-k\delta+1],   |S_{n-1}|\leq x^{\varepsilon_0}\right)\nonumber\\
	& \lesssim_\beta   \beta^{(\log x)^{5/2}/2}   \sup_{|z| \leq x^{\varepsilon_0} } \widetilde{\mathbb{P}}(X\in (x-z-k\delta, x-z-k\delta+1))   \nonumber\\
	&\leq \sup_{|z|\leq 4x^{\varepsilon_0} +1} \widetilde{\mathbb{P}}(X\in (x-z-1, x-z)) .
\end{align}
According to standard analysis, we have 
\begin{align}\label{e26-4}
	& \frac{e^{\lambda x^b}}{\ell(x)x^a } \widehat{T}_{33} \lesssim_\beta   \frac{e^{\lambda x^b}}{\ell(x)x^a } \sup_{|z| \leq 4x^{\varepsilon_0} +1}  \int_{x-z-1}^{x-z} \frac{\ell(y)}{\mathbb{E}(e^{\gamma X})} y^a e^{-\lambda y^b}\mathrm{d} y \nonumber\\
	& \lesssim e^{\lambda x^b-\lambda(x-6x^{\varepsilon_0}-2)^b} \lesssim 1 .
\end{align}
Combining \eqref{decomposition-T-3}, \eqref{e26-2}, \eqref{e30-1} and \eqref{e26-4} we get (i).

If $I= [p, q)$, let $x$ be sufficiently large such that $q\leq x^{\varepsilon_0}$. Then by Lemma \ref{lem1},
\begin{align}\label{e29-2}
	& \limsup_{x\to+\infty} \frac{e^{\lambda x^b}}{\ell(x)x^a }  \widehat{T}_{33} \nonumber\\
	&\leq     \limsup_{x\to+\infty}  \frac{e^{\lambda x^b}}{\ell(x)x^a } \sum_{(\log x)^{5/2}\leq n<x^{(1+b)/2}}  n \beta^{n} \widetilde{\mathbb{P}}\left(S_n>x-q,   |S_{n-1}|\leq x^{\varepsilon_0}\right)\nonumber\\
	& \lesssim_\beta  \limsup_{x\to+\infty}  \frac{e^{\lambda x^b}\beta^{(\log x)^{5/2}/2}  }{\ell(x)x^a }  \sup_{|z| \leq 2x^{\varepsilon_0} } \widetilde{\mathbb{P}}(X>x-z)    \lesssim \limsup_{x\to+\infty} \beta^{(\log x)^{5/2}/2} x^{1-b}  \nonumber\\
	&= 0. 
\end{align}
Combining  \eqref{decomposition-T-3}, \eqref{e26-2'}, \eqref{e30-1} and \eqref{e29-2}, we get (ii).

\hfill$\Box$

Next we treat $T_4$. We have the following decomposition:
\begin{align}\label{decomposition-T-4}
	&T_4= \sum_{n< (\log x)^{5/2}}  \beta^{n}\widetilde{\mathbb{P}}\left(\max_{1\leq j \leq n} S_j \leq x, x-S_n\in I,  \max_{1\leq i\leq n} X_i \leq 2x^{\varepsilon_0}\right) \nonumber\\
	&\qquad + \sum_{n< (\log x)^{5/2}}  \beta^{n}\widetilde{\mathbb{P}}\left(\max_{1\leq j \leq n} S_j \leq x, x-S_n\in I, 2x^{\varepsilon_0}<\max_{1\leq i\leq n} X_i  \leq x-2x^{\varepsilon_0}\right)\nonumber\\
	&\qquad +  \sum_{n< (\log x)^{5/2}}  \beta^{n}  \widetilde{\mathbb{P}}\left(\max_{1\leq j \leq n} S_j \leq x, x-S_n\in I,  \exists !\ 1\leq j\leq n,\  X_j > x-2x^{\varepsilon_0}\right)\nonumber\\
	&=: T_{41}+T_{42}+T_{43}. 
\end{align}
Noticing that uniformly for $x\geq k\delta >x- 2(\log x)^{5/2} x^{\varepsilon_0}$,  $\lim_{x\to+\infty} (x^b- (k\delta)^b)= 0$ and this implies that 
\begin{align}\label{e34-1}
	& \frac{e^{\lambda x^b}}{\ell(x)x^a } T_{4} \leq   \frac{e^{\lambda x^b}}{\ell(x)x^a }  \sum_{n< (\log x)^{5/2}}  \beta^{n}   \lesssim_\beta  \frac{e^{\lambda x^b}}{\ell(k\delta) (k\delta)^a }   \lesssim \frac{e^{\lambda (k\delta)^b}}{\ell(k\delta) (k\delta)^a}.
\end{align}

\begin{lemma}\label{lem5-3}
	Assume {\bf(H2)} and $b\in (0,1)$. 
	\begin{itemize}
		\item[(i)] For all $k\delta\geq K_0, I=[(k-1)\delta, k\delta)$ and $x>K_0$, 
		\begin{align}
			\frac{e^{\lambda x^b}}{\ell(x)x^a } T_{41} \lesssim_\beta  \frac{e^{\lambda (k\delta)^b}}{\ell(k\delta) (k\delta)^a}.
		\end{align}
		\item[(ii)] For each $I=[p,q)$, we have 
		\begin{align}
			\lim_{x\to+\infty}\frac{e^{\lambda x^b}}{\ell(x)x^a } T_{41}  =0. 
		\end{align}
	\end{itemize}
\end{lemma}
\textbf{Proof: }  Combining Lemma \ref{lem5-1} (ii), \eqref{decomposition-T-4} and \eqref{e34-1}, it suffices to consider the case $1\leq k\delta \leq x-2(\log x)^{5/2} x^{\varepsilon_0}$ in the proof of (i). Set $z=k\delta$ if $I=[(k-1)\delta, k\delta)$ and $z=q$ if $I=[p,q)$. Then for $x$ large enough such that   
$z \leq x-2(\log x)^{5/2} x^{\varepsilon_0}$, we have 
\begin{align}\label{e34-2}
	&T_{41} \leq \sum_{n<(\log x)^{5/2}}  \beta^{n}\widetilde{\mathbb{P}}\left(  S_n>x- z,   S_n \leq 2nx^{\varepsilon_0}\right) \nonumber\\
	& \leq \sum_{n<(\log x)^{5/2}}  \beta^{n} \widetilde{\mathbb{P}}\left(  S_n>2(\log x)^{5/2} x^{\varepsilon_0},   S_n \leq 2(\log x)^{5/2}x^{\varepsilon_0}\right)=0,
\end{align}
which completes the proof of the lemma. 

\hfill$\Box$

\begin{lemma}\label{lem5-4}
	Assume {\bf(H2)} and $b\in (0,1)$. 
	\begin{itemize}
		\item[(i)] For all $k\delta\geq K_0, I=[(k-1)\delta, k\delta)$ and $x>K_0$, 
		\begin{align}
			\frac{e^{\lambda x^b}}{\ell(x)x^a } T_{42} \lesssim_\beta  \frac{e^{\lambda (k\delta)^b}}{\ell(k\delta) (k\delta)^a}.
		\end{align}
		\item[(ii)] For each $I=[p,q)$, we have 
		\begin{align}
			\lim_{x\to+\infty}\frac{e^{\lambda x^b}}{\ell(x)x^a } T_{42}  =0. 
		\end{align}
	\end{itemize}
\end{lemma}
\textbf{Proof: }  We divide the proof into two steps.

\noindent
{\bf(Step 1)}. In this step, we prove (ii). Noticing that 
\begin{align}
	T_{42} & \leq \sum_{n< (\log x)^{5/2}}  n\beta^{n}\widetilde{\mathbb{P}}\left( S_{n-1}>x-q-X_n,   2x^{\varepsilon_0}< X_n \leq x-2x^{\varepsilon_0}\right) \nonumber\\
	& = \sum_{n< (\log x)^{5/2}}  n\beta^{n}\widetilde{\mathbb{E}}\left(1_{\{ 2x^{\varepsilon_0}< X_n \leq x-2x^{\varepsilon_0}\}}  \widetilde{\mathbb{P}} (S_{n-1}>y)|_{y= x-q-X_n}   \right).
\end{align}
Let $x$ be sufficiently large such that $x^{\varepsilon_0}>q$, then  $y=x-q-X_{n} > x^{\varepsilon_0}$ on $\{ 2x^{\varepsilon_0}< X_n \leq x-2x^{\varepsilon_0}\} $. Therefore, by \eqref{e6}, for $x$ large enough such that $ (\log x^{\varepsilon_0})^3> (\log x)^{5/2}$, we have for $x$ large enough, 
\begin{align}
	&T_{42}\leq 2\sum_{n< (\log x)^{5/2}}  n^2\beta^{n}\widetilde{\mathbb{E}}\left(1_{\{ 2x^{\varepsilon_0}< X_n \leq x-2x^{\varepsilon_0}\}}  \widetilde{\mathbb{P}} (X>y)|_{y= x-q-X_n}   \right).
\end{align}
Combining Lemma \ref{lem1} and inequalities $\ell(x)x^a \gtrsim x^{-|a|-1}$ and  $\widetilde{\ell}(x)x^{a+1-b}\lesssim x^{|a|+2-b}$, we deduce that 
\begin{align}\label{e34-3}
	& \limsup_{x\to+\infty} \frac{e^{\lambda x^b}}{\ell(x)x^a } T_{42} \nonumber\\
	& \lesssim \limsup_{x\to+\infty} e^{\lambda x^b}x^{|a|+1}\sum_{n< (\log x)^{5/2}}  n^2\beta^{n}\widetilde{\mathbb{E}}\left(1_{\{ 2x^{\varepsilon_0}< X_n \leq x-2x^{\varepsilon_0}\}} (x-q-X_n)^{|a|+2-b} e^{-\lambda (x-q-X_n)^b}   \right) \nonumber\\
	&\lesssim_\beta  \limsup_{x\to+\infty} e^{\lambda x^b}x^{2|a|+3-b} \widetilde{\mathbb{E}}\left(1_{\{ 2x^{\varepsilon_0}< X \leq x-2x^{\varepsilon_0}\}}   e^{-\lambda (x-q-X)^b}   \right) \nonumber\\
	&\lesssim e^{\lambda q^b} \limsup_{x\to+\infty} e^{\lambda x^b}x^{2|a|+3-b} \widetilde{\mathbb{E}}\left(1_{\{ 2x^{\varepsilon_0}< X \leq x-2x^{\varepsilon_0}\}}   e^{-\lambda (x-X)^b}   \right) ,
\end{align}
where in the last inequality we used inequality $(u+v)^b\leq u^b+ v^b$ for $u,v>0$ and $b\in (0,1)$.  Combining \eqref{e10-1} (with $a=0$ and noticing that $|x-X|\geq 2x^{\varepsilon_0}$ in on $\{ 2x^{\varepsilon_0}< X \leq x-2x^{\varepsilon_0}\}$) and \eqref{e34-3}, we get (ii). 

\noindent
{\bf(Step 2)}. In this step, we prove (i). Combining Lemma \ref{lem5-1} (ii), \eqref{decomposition-T-4} and  \eqref{e34-1}, it suffices to consider the case $1\leq k\delta \leq x- 2(\log x)^{5/2} x^{\varepsilon_0}$.
We have the upper bound 
\begin{align}\label{decomposition-T-42}
	&T_{42}\leq \sum_{n< (\log x)^{5/2}}  n\beta^{n}\widetilde{\mathbb{P}}\left( x-S_n\in I,      2x^{\varepsilon_0}< X_n \leq x-2x^{\varepsilon_0}\right) \nonumber\\
	&\leq \sum_{n< (\log x)^{5/2}}  n\beta^{n}\widetilde{\mathbb{P}}\left( x-S_n<k\delta,    x-k\delta-X_n>x^{\varepsilon_0},   2x^{\varepsilon_0}< X_n \leq x-2x^{\varepsilon_0}\right) \nonumber\\
	&\qquad + \sum_{n< (\log x)^{5/2}}  n\beta^{n}\widetilde{\mathbb{P}}\left( x-S_n\in I,    x-k\delta-X_n\leq x^{\varepsilon_0},   2x^{\varepsilon_0}< X_n \leq x-2x^{\varepsilon_0}\right) \nonumber\\
	&=: T_{42}' + T_{42}''. 
\end{align}

\noindent
{\bf Estimate for $T_{42}'$}.  According to the independence between $S_{n-1}$ and $X_n$, we have 
\begin{align}
	&T_{42}' = \sum_{n< (\log x)^{5/2}}  n\beta^{n}\widetilde{\mathbb{E}}\left(1_{\{x-k\delta-X_n>x^{\varepsilon_0},   2x^{\varepsilon_0}< X_n \leq x-2x^{\varepsilon_0} \}}  \widetilde{\mathbb{P}}(S_{n-1}>y)|_{y= x-k\delta-X_n}    \right) . 
\end{align}
Similar to \eqref{e34-3}, let $x$ be large enough such that $ (\log x^{\varepsilon_0})^3> (\log x)^{5/2}$,  it follows from Lemma \ref{lem1} and inequalities $\ell(x)x^a \gtrsim x^{-|a|-1}$ and  $\widetilde{\ell}(x)x^{a+1-b}\lesssim x^{|a|+2-b}$ that 
\begin{align}\label{e28-1}
	& \frac{e^{\lambda x^b} }{\ell(x)x^a } T_{42}'\nonumber\\
	&\lesssim  x^{|a|+1} e^{\lambda x^b}\sum_{n< (\log x)^{5/2}}  n^2\beta^{n}\widetilde{\mathbb{E}}\left(1_{\{x-k\delta-X_n>x^{\varepsilon_0},   2x^{\varepsilon_0}< X_n \leq x-2x^{\varepsilon_0} \}}  \widetilde{\mathbb{P}}(X>y)|_{y= x-k\delta-X_n}    \right) \nonumber\\
	&\lesssim_\beta  x^{|a|+1} e^{\lambda x^b}\widetilde{\mathbb{E}}\left(1_{\{x-k\delta-X_n>x^{\varepsilon_0},   2x^{\varepsilon_0}< X_n \leq x-2x^{\varepsilon_0} \}}  ( x-k\delta-X_n)^{|a|+2-b} e^{-\lambda ( x-k\delta-X_n)^b}   \right) \nonumber\\
	&\lesssim x^{2|a|+3-b} e^{\lambda x^b}\widetilde{\mathbb{E}}\left(1_{\{x-k\delta- x^{\varepsilon_0} >x-k\delta-X>x^{\varepsilon_0}    \}}   e^{-\lambda ( x-k\delta-X)^b}   \right).
\end{align}
Since $x\geq x-k\delta-j \geq x^{\varepsilon_0}-1$ for all $j\leq \lceil x-k\delta -x^{\varepsilon_0}\rceil$,  it follows from inequality $\ell(x)\lesssim x$ and \eqref{density-X-tildeP} that 
\begin{align}
	& \widetilde{\mathbb{E}}\left( 1_{\{x-k\delta- x^{\varepsilon_0} >x-k\delta-X>x^{\varepsilon_0}    \}}    e^{-\lambda (x-k\delta-X)^b}\right) \nonumber\\
	& \leq   \sum_{j= \lfloor x^{\varepsilon_0}\rfloor}^{\lceil x-k\delta -x^{\varepsilon_0}\rceil}  e^{-\lambda j^b} \widetilde{\mathbb{P}}(x-k\delta-X \in [j, j+1]) \nonumber\\
	& \lesssim \sum_{j= \lfloor x^{\varepsilon_0}\rfloor}^{\lceil x-k\delta -x^{\varepsilon_0}\rceil}  e^{-\lambda j^b} (x-k\delta-j)^{|a|+1}e^{-\lambda (x-k\delta -j)^b} \nonumber\\
	&\lesssim  x^{|a|+1}\sum_{j= \lfloor x^{\varepsilon_0}\rfloor}^{\lceil x-k\delta -x^{\varepsilon_0}\rceil}  e^{-\lambda j^b -\lambda (x-k\delta -j)^b} .
\end{align}
Plugging this back to \eqref{e28-1} yields that 
\begin{align}\label{e27-4}
	& \frac{e^{\lambda x^b} }{\ell(x)x^a } T_{42}' \lesssim_\beta  x^{3|a|+4-b} e^{\lambda x^b}\sum_{j= \lfloor x^{\varepsilon_0}\rfloor}^{\lceil x-k\delta -x^{\varepsilon_0}\rceil}  e^{-\lambda j^b -\lambda (x-k\delta -j)^b} .
\end{align}
Noticing that  for $\lfloor x^{\varepsilon_0}\rfloor\leq j\leq \lceil x-k\delta -x^{\varepsilon_0}\rceil$, we have for large $x$, 
\begin{align}
	& j^b+ (x-j-k\delta)^b \nonumber\\
	& \geq \min\{ (\lfloor x^{\varepsilon_0}\rfloor)^b+ (x-\lfloor x^{\varepsilon_0}\rfloor-k\delta)^b , (\lceil x-k\delta -x^{\varepsilon_0}\rceil)^b + (x-k\delta -\lceil x-k\delta -x^{\varepsilon_0}\rceil)^b\} \nonumber\\
	&\geq ( x^{\varepsilon_0} )^b+ (x-  x^{\varepsilon_0} -k\delta)^b -1.
\end{align}
Plugging this back to \eqref{e27-4} and using the fact that $x^b-(x-x^{\varepsilon_0})^b \to 0$,  we have
\begin{align}
	&\frac{e^{\lambda x^b} }{\ell(x)x^a } T_{42}'  \lesssim_\beta x^{3|a|+4-b} e^{\lambda x^b}  \sum_{j= \lfloor x^{\varepsilon_0}\rfloor}^{\lceil x-k\delta -x^{\varepsilon_0}\rceil} e^{-\lambda ( x^{\varepsilon_0} )^b- \lambda (x-  x^{\varepsilon_0} -k\delta)^b}\nonumber\\
	&\lesssim x^{3|a|+5-b} e^{\lambda (x-x^{\varepsilon_0})^b -\lambda ( x^{\varepsilon_0} )^b- \lambda (x-  x^{\varepsilon_0} -k\delta)^b}.
\end{align}
Noticing that for $b\in (0,1)$,   $u^b+v^b \geq (u+v)^b$ for all $u,v>0$. Therefore, applying $u=k\delta, v= x-x^{\varepsilon_0}-k\delta$ in the above inequality, we deduce that for all $k\delta\leq x-2(\log x)^{5/2} x^{\varepsilon_0}$ and large $x$, 
\begin{align}\label{e28-3}
	& \frac{e^{\lambda x^b} }{\ell(x)x^a } T_{42}' \lesssim_\beta x^{3|a|+5-b} e^{\lambda (k\delta)^b -\lambda ( x^{\varepsilon_0} )^b } \lesssim \frac{x^{4|a|+6-b}}{\ell(k\delta) (k\delta)^a}e^{\lambda (k\delta)^b -\lambda ( x^{\varepsilon_0} )^b } \lesssim \frac{e^{\lambda (k\delta)^b }}{\ell(k\delta) (k\delta)^a}.
\end{align}

\noindent
{\bf Estimate for $T_{42}''$}.  Since $k\delta \leq x-2(\log x)^{5/2} x^{\varepsilon_0}< x-3x^{\varepsilon_0}$ for large $x$, we can rewrite $T_{42}''$ by 
\begin{align}\label{decom-T-42''}
	&T_{42}'' = \sum_{n< (\log x)^{5/2}}  n\beta^{n}\widetilde{\mathbb{P}}\left( x-S_n\in I,    x-k\delta-x^{\varepsilon_0}\leq X_n     \leq x-2x^{\varepsilon_0}\right)=:R_{42}''+S_{42}'',
\end{align}
where 
\begin{align}
	R_{42}'' & :=\sum_{n< (\log x)^{5/2}}  n\beta^{n}\widetilde{\mathbb{P}}\left( x-S_n\in I,  S_{n-1}<-x^{\varepsilon_0},  x-k\delta-x^{\varepsilon_0}\leq X_n     \leq x-2x^{\varepsilon_0}\right),\\
	Q_{42}''&:= \sum_{n< (\log x)^{5/2}}  n\beta^{n}\widetilde{\mathbb{P}}\left( x-S_n\in I,  S_{n-1}\geq -x^{\varepsilon_0},  x-k\delta-x^{\varepsilon_0}\leq X_n     \leq x-2x^{\varepsilon_0}\right).
\end{align}

\noindent
{\bf Estimate for $R_{42}''$.}
Combining inequality $\ell(x)\lesssim x$ Lemma \ref{lem1} and \eqref{left-tail-X}, 
\begin{align}
	& \frac{e^{\lambda x^b}}{\ell(x)x^a } R_{42}''  \leq \sum_{n< (\log x)^{5/2}}  n^2 \beta^{n}   \widetilde{\mathbb{P}}\left( X< -x^{{\varepsilon_0}}/n \right) \widetilde{\mathbb{P}}(X\geq x-k\delta-x^{\varepsilon_0}) \nonumber\\
	& \lesssim_\beta   (x-k\delta-x^{\varepsilon_0})^{a+2-b}e^{-\lambda (x-k\delta-x^{\varepsilon_0})^b} \widetilde{\mathbb{P}}\left( X< -x^{\varepsilon_0}/(\log x)^{5/2} \right)\nonumber\\
	&\lesssim e^{-\gamma x^{\varepsilon_0}/(\log x)^{5/2}} x^{|a|+2-b} e^{-\lambda (x-k\delta-x^{\varepsilon_0})^b},
\end{align}
which together with inequality $e^{\lambda x^b}\lesssim e^{\lambda (x-x^{\varepsilon_0})^b}\leq  e^{\lambda (x-x^{\varepsilon_0}-k\delta)^b+ \lambda(k\delta)^b}$  implies that 
\begin{align}\label{e28-7}
	& \frac{e^{\lambda x^b}}{\ell(x)x^a } R_{42}''    \lesssim_\beta  \frac{e^{\lambda x^b}x^{|a|+2-b}}{\ell(x)x^a }  e^{-\gamma x^{\varepsilon_0}/(\log x)^{5/2} -\lambda (x-k\delta-x^{\varepsilon_0})^b} \nonumber\\
	&\lesssim \frac{e^{\lambda (k\delta)^b}x^{|a|+2-b}}{\ell(x)x^a }  e^{-\gamma x^{\varepsilon_0}/(\log x)^{5/2}}   \nonumber\\
	&\lesssim \frac{e^{\lambda (k\delta)^b}}{\ell(k\delta)(k\delta)^a }. 
\end{align}

\noindent
{\bf Estimate for $Q_{42}''$.}
Noticing that on $\{x-S_n\in I\}\cap \{X_n\geq x-k\delta-x^{\varepsilon_0}\}$, we have 
\[
S_{n-1}\leq x-k\delta +1-X_n \leq 2x^{\varepsilon_0}.
\]
Thus, this combined with the event $\{S_{n-1}\geq -x^{\varepsilon_0}\}$ implies that 
\begin{align}\label{e28-4}
	Q_{42}'' \leq   \sum_{n< (\log x)^{5/2}}  n\beta^{n} \sup_{|z|\leq 2x^{\varepsilon_0}} \widetilde{\mathbb{P}}\left( X\in (x-k\delta-z, x-k\delta-z +1]\right).
\end{align}

\noindent
{\bf Case 1: $k\delta< 4x^{\varepsilon_0}$.}
In this case, combining \eqref{density-X-tildeP} and \eqref{e28-4}, we have 
\begin{align}\label{e28-5}
	& \frac{e^{\lambda x^b}}{\ell(x)x^a } Q_{42}''  \lesssim_\beta  \frac{e^{\lambda x^b}}{\ell(x)x^a }   \sup_{|z|\leq 6x^{\varepsilon_0}}  \ell(x-z)(x-z)^a e^{-\lambda (x-z)^b}  \lesssim 1 \lesssim \frac{e^{\lambda (k\delta)^b}}{\ell(k\delta)(k\delta)^a }.
\end{align}

\noindent
{\bf Case 2: $4x^{\varepsilon_0} \leq k\delta \leq x-2(\log x)^{5/2} x^{\varepsilon_0}<x-4x^{\varepsilon_0}$.}
In this case, it follows from \eqref{e28-4} that
\begin{align}
	&\frac{e^{\lambda x^b}}{\ell(x)x^a } Q_{42}''   \lesssim_\beta   x^{|a|+1} e^{\lambda x^b} \sup_{|z|\leq 2x^{\varepsilon_0}}  \ell(x-k\delta-z)(x-k\delta-z)^a e^{-\lambda (x-k\delta-z)^b}\nonumber\\
	&\lesssim x^{|a|+1} e^{\lambda x^b} \sup_{|z|\leq 2x^{\varepsilon_0}}   (x-k\delta-z)^{a+1} e^{-\lambda (x-k\delta-z)^b} \nonumber\\
	&\lesssim \frac{x^{3|a|+3} }{\ell(k\delta)(k\delta)^a}e^{\lambda x^b}  e^{-\lambda (x-k\delta-2x^{\varepsilon_0})^b}.
\end{align}
Since for $y\in[4x^{\varepsilon_0}, x-4x^{\varepsilon_0}]$, when $x$ is large enough,
\begin{align}
	 & y^b+ (x-y-2x^{\varepsilon_0})^b  \geq \min\{ (4x^{\varepsilon_0})^b+ (x-6x^{\varepsilon_0})^b, (x-4x^{\varepsilon_0})^b+ (2x^{\varepsilon_0})^b\} \nonumber\\
	& \geq x^b +(2x^{\varepsilon_0})^b-1.
\end{align}
Therefore, we conclude that 
\begin{align}\label{e28-6}
	& \frac{e^{\lambda x^b}}{\ell(x)x^a } Q_{42}''  \lesssim_\beta  \frac{x^{3|a|+3} }{\ell(k\delta)(k\delta)^a}e^{\lambda (k\delta)^b}  e^{-\lambda ( 2x^{\varepsilon_0})^b}\lesssim \frac{e^{\lambda (k\delta)^b}}{\ell(k\delta)(k\delta)^a }.
\end{align}
Combining \eqref{decom-T-42''}, \eqref{e28-7}, \eqref{e28-5} and \eqref{e28-6}, it holds that 
\begin{align}\label{e29-1}
	\frac{e^{\lambda x^b}}{\ell(x)x^a } T_{42}''   = \frac{e^{\lambda x^b}}{\ell(x)x^a } (R_{42}''+Q_{42}'')  \lesssim_\beta  \frac{e^{\lambda (k\delta)^b}}{\ell(k\delta)(k\delta)^a }.
\end{align}
Combining  \eqref{decomposition-T-42},  \eqref{e28-3} and \eqref{e29-1}, we complete the proof of (i). 

\hfill$\Box$

\noindent
\textbf{Proof of Proposition \ref{prop3} for $b\in (0,1)$:}
Combining \eqref{decomposition}, Lemmas \ref{lem5-1} and \ref{lem5-2}, to prove (i), it suffices to prove that for all $K_0\leq k\delta \leq x, I= [(k-1)\delta, k\delta)$ and $x>K_0$,
\begin{align}\label{Goal-1}
	\frac{ e^{\lambda x^b}}{\ell(x)x^a} T_4\lesssim_\beta \frac{e^{\lambda (k\delta)^b}}{\ell(k\delta) (k\delta)^a}.
\end{align}
Also, to prover (ii), it suffices to show that for each $I=[p,q)$,
\begin{align}\label{Goal-2}
	\lim_{x\to+\infty} \frac{ e^{\lambda x^b}}{\ell(x)x^a} T_4 = \frac{1}{\mathbb{E}(e^{\gamma X})} \int_p^q \sum_{j=0}^\infty \frac{\beta^{j+1}}{1-\beta} \widetilde{\mathbb{P}}\left(z> -\min_{0\leq k\leq j} S_k \right)\mathrm{d} z. 
\end{align}
Combining \eqref{decomposition-T-4}, \eqref{e34-1}, Lemmas \ref{lem5-3} and \ref{lem5-4}, to prove \eqref{Goal-1}, it suffices to prove that for all 
$K_0\leq k\delta \leq x-2(\log x)^{5/2}x^{\varepsilon_0}, I= [(k-1)\delta, k\delta)$ and $x>K_0$,
\begin{align}\label{Goal-1'}
	\frac{ e^{\lambda x^b}}{\ell(x)x^a} T_{43}\lesssim_\beta  \frac{e^{\lambda (k\delta)^b}}{\ell(k\delta) (k\delta)^a}.
\end{align}
Also, to prover \eqref{Goal-2}, it suffices to show that for each $I=[p,q)$,
\begin{align}\label{Goal-2'}
	\lim_{x\to+\infty} \frac{ e^{\lambda x^b}}{\ell(x)x^a} T_{43} = \frac{1}{\mathbb{E}(e^{\gamma X})}  \int_p^q \sum_{j=0}^\infty \frac{\beta^{j+1}}{1-\beta} \widetilde{\mathbb{P}}\left(z> -\min_{0\leq k\leq j} S_k \right)\mathrm{d} z. 
\end{align}
{\bf(Step 1)}. In this step, we prove \eqref{Goal-1'}.  
According to the definition of $T_{43}$, we have the upper bound 
\begin{align}\label{decomposition-T-43}
	&T_{43}\leq \sum_{n< (\log x)^{5/2}}  n\beta^{n}  \widetilde{\mathbb{P}}\left(S_n\in [x-k\delta, x-k\delta+1],  X_n > x-2x^{\varepsilon_0}\right)\nonumber\\
	&\leq \sum_{n< (\log x)^{5/2}}  n\beta^{n}  \widetilde{\mathbb{P}}\left( |S_{n-1}|>x^{\varepsilon_0},  X_n > x-2x^{\varepsilon_0}\right) \nonumber\\
	&\qquad +\sum_{n< (\log x)^{5/2}}  n\beta^{n}  \sup_{|z|\leq x^{\varepsilon_0}} \widetilde{\mathbb{P}}\left( X\in[x-k\delta-z, x-k\delta-z+1],  X > x-2x^{\varepsilon_0}\right)\nonumber\\
	&=: T_{43}'+T_{43}''.
\end{align}

\noindent
{\bf Estimate for $T_{43}'$.} Similar to \eqref{e30-1},  taking $k_0=1/\varepsilon_0$, it follows from   Markov's inequality and Lemma \ref{lem1} that 
\begin{align}\label{upp-L-1}
	&\frac{ e^{\lambda x^b}}{\ell(x)x^a} T_{43}' \leq \frac{\widetilde{\mathbb{P}}\left(  X > x-2x^{\varepsilon_0}\right) }{\ell(x)x^a e^{-\lambda x^b}} \sum_{n< (\log x)^{5/2}}  n\beta^{n}  \widetilde{\mathbb{P}}\left( |X|>x^{\varepsilon_0}/n\right) \nonumber\\
	&\lesssim \frac{\ell(x)x^{a+1-b}e^{-\lambda x^b}}{\ell(x)x^a e^{-\lambda x^b}} \sum_{n< (\log x)^{5/2}}  n\beta^{n}  \frac{\widetilde{\mathbb{E}}\left(|X|^{k_0} \right) }{(x^{\varepsilon_0}/n)^{k_0} }  \nonumber\\
	&\lesssim_\beta  x^{-b} \lesssim \frac{e^{\lambda (k\delta)^b}}{\ell(k\delta)(k\delta)^a }.
\end{align}

\noindent
{\bf Estimate for$T_{43}''$.} Noticing that $T_{43}''=0$ when $k\delta>3x^{\varepsilon_0}$, so we only consider the case $k\delta \leq 3x^{\varepsilon_0}$.  In this case, we have
\begin{align}\label{upp-L-2}
	&\frac{ e^{\lambda x^b}}{\ell(x)x^a} T_{43}''  \lesssim_\beta  \frac{ e^{\lambda x^b}}{\ell(x)x^a}    \sup_{|z|\leq 4x^{\varepsilon_0}} \widetilde{\mathbb{P}}\left( X\in[x-z, x-z+1]\right) \nonumber\\
	&\lesssim \frac{ e^{\lambda x^b}}{\ell(x)x^a}   \sup_{|z|\leq 4x^{\varepsilon_0}} \ell(x-z) (x-z)^a e^{-\lambda (x-z)^b}\lesssim 1\nonumber\\
	& \lesssim \frac{e^{\lambda (k\delta)^b}}{\ell(k\delta)(k\delta)^a }.
\end{align}
Combining \eqref{decomposition-T-43}, \eqref{upp-L-1} and \eqref{upp-L-2}, we get \eqref{Goal-1'}.

\noindent
{\bf(Step 2)}. In this step, we prove  \eqref{Goal-2'}.   We have the following decomposition: 
\begin{align}\label{decomposition-T-43-QR}
	&T_{43} = \sum_{n< (\log x)^{5/2}}  \beta^{n}  \widetilde{\mathbb{P}}\left(\max_{1\leq j \leq n} S_j \leq x, x-S_n\in I,  \exists !\ 1\leq j\leq n,\  X_j > x-2x^{\varepsilon_0}, \max_{i\neq j} X_i \leq x^{\varepsilon_0}\right) \nonumber\\
	&\quad + \sum_{n< (\log x)^{5/2}}  \beta^{n}  \widetilde{\mathbb{P}}\left(\max_{1\leq j \leq n} S_j \leq x, x-S_n\in I,  \exists !\ 1\leq j\leq n,\  X_j > x-2x^{\varepsilon_0}, \max_{i\neq j} X_i > x^{\varepsilon_0}\right) \nonumber\\
	&=: R_{43}+Q_{43}. 
\end{align}

\noindent
{\bf Estimate for $Q_{43}$.}  By Lemma \ref{lem1}, 
\begin{align}\label{e35-1}
	& \frac{e^{\lambda x^b}}{\ell(x)x^a } Q_{43}  \leq \frac{e^{\lambda x^b}}{\ell(x)x^a }  \sum_{n< (\log x)^{5/2}} n \beta^{n}  \widetilde{\mathbb{P}}\left(   X_n > x-2x^{\varepsilon_0}, \max_{i\leq n-1} X_i > x^{\varepsilon_0}\right) \nonumber\\
	&\leq \frac{e^{\lambda x^b}}{\ell(x)x^a }  \sum_{n< (\log x)^{5/2}} n^2 \beta^{n}  \widetilde{\mathbb{P}}\left(   X > x-2x^{\varepsilon_0} \right) \widetilde{\mathbb{P}}\left(   X > x^{\varepsilon_0}\right) \nonumber\\
	&\lesssim_\beta  \frac{e^{\lambda x^b}}{\ell(x)x^a }  \widetilde{\mathbb{P}}\left(   X > x-2x^{\varepsilon_0} \right) \widetilde{\mathbb{P}}\left(   X > x^{\varepsilon_0}\right) \lesssim  x^{1-b} \ell(x^{\varepsilon_0}) (x^{\varepsilon_0})^{a+1-b}e^{-\lambda (x^{\varepsilon_0})^b} \nonumber\\
	&\stackrel{x\to+\infty}{\longrightarrow}0.
\end{align}

\noindent
{\bf Estimate for $R_{43}$}. We have 
\begin{align}\label{e36-2}
	& \lim_{x\to+\infty} \frac{e^{\lambda x^b}}{\ell(x)x^a } R_{43} \nonumber\\
	& = \lim_{x\to+\infty} \frac{e^{\lambda x^b}}{\ell(x)x^a }  \sum_{n<(\log x)^{5/2}}  \beta^{n} \sum_{j=1}^n \nonumber\\
	&\qquad \times \widetilde{\mathbb{P}}\left(\max_{1\leq k \leq n} S_k\leq x, x-S_n\in I,   X_j>x-2x^{\varepsilon_0},\max_{i\neq j} X_i \leq x^{\varepsilon_0} \right). 
\end{align}
According to a similar argument leading to \eqref{e30-1}, we obtain that 
\begin{align}
	&  \lim_{x\to+\infty} \frac{e^{\lambda x^b}}{\ell(x)x^a } R_{43} \nonumber\\
	& =\lim_{x\to+\infty}  \frac{e^{\lambda x^b}}{\ell(x)x^a }  \sum_{n<(\log x)^{5/2}}  \beta^{n} \sum_{j=1}^n \nonumber\\
	&\qquad \times \widetilde{\mathbb{P}}\left(\max_{1\leq k \leq n} S_k \leq x, x-S_n\in I,     |S_n-X_j| \leq x^{\varepsilon_0} ,  \max_{i\neq j} X_i\leq x^{\varepsilon_0} \right). 
\end{align}
Set $W=S_n-\max_{j\leq k\leq n}S_k$ and $Y=x- (S_n-X_j)$, then by the independence among $\{X_i, 1\leq i\leq n\}$,
\begin{align}
	& \widetilde{\mathbb{P}}\left(\max_{1\leq j \leq n} S_j \leq x, x-S_n\in I,     |S_n-X_j| \leq x^{\varepsilon_0} ,  \max_{i\neq j} X_i\leq x^{\varepsilon_0} \right)\nonumber\\
	&= \widetilde{\mathbb{P}}\left(X_j \leq W+Y, X_j\in (Y-q, Y-p],    |x-Y| \leq x^{\varepsilon_0} ,  \max_{i\neq j} X_i\leq x^{\varepsilon_0} \right)\nonumber\\
	&=\frac{1}{\mathbb{E}(e^{\gamma X})} \widetilde{\mathbb{E}}\left(1_{\{|x-Y| \leq  x^{\varepsilon_0} ,  \max_{i\neq j} X_i\leq x^{\varepsilon_0}, W>-q \}}  \int_{ -q }^{\min\{W, -p\}} \ell(z+Y)(z+Y)^a e^{-\lambda (z+Y)^b}\mathrm{d}z \right).
\end{align}
Since on the event $|x-Y|\leq x^{\varepsilon_0}$, for all $z\in [-q, -p]$, 
$\frac{1}{\ell(x)x^a e^{-\lambda x^b}}\ell(z+Y)(z+Y)^a e^{-\lambda (z+Y)^b}$ converges uniformly to $1$. Therefore, we conclude that 
\begin{align}\label{e35-2}
	&  \lim_{x\to+\infty} \frac{e^{\lambda x^b}}{\ell(x)x^a } R_{43} \nonumber\\
	& =\frac{1}{\mathbb{E}(e^{\gamma X})} \lim_{x\to+\infty} \sum_{n<(\log x)^{5/2}}  \beta^{n} \sum_{j=1}^n\widetilde{\mathbb{E}}\left(1_{\{ |x-Y| \leq  x^{\varepsilon_0} ,  \max_{i\neq j} X_i\leq 2x^{\varepsilon_0}\}}  \int_{ -q }^{-p} 1_{\{z<W\}}\mathrm{d}z \right)\nonumber\\
	& = \frac{1}{\mathbb{E}(e^{\gamma X})}    \int_{ -q }^{-p}\sum_{n=1}^\infty  \beta^{n} \sum_{j=1}^n\widetilde{\mathbb{P}}\left(   z<W\right) \mathrm{d}z .
\end{align}
Combining \eqref{decomposition-T-43-QR}, \eqref{e35-1} and \eqref{e35-2}, we complete the proof of \eqref{Goal-2'}. 

\hfill$\Box$

\bigskip
\noindent


	\begin{singlespace}
		
	\end{singlespace}

	\vskip 0.2truein
	\vskip 0.2truein

\noindent{\bf Haojie Hou:}  School of Statistics and Data Science, Nankai University, Tianjin 300071, P. R. China.  Email: {\texttt houhaojie@nankai.edu.cn}

\end{document}